\documentclass[11pt,reqno]{amsart}
\usepackage{amssymb, bbm, color}
\usepackage[makeroom]{cancel}
\usepackage{ulem}
\usepackage{amsfonts}
\usepackage{mathrsfs}
\usepackage{amsmath}
\usepackage{graphicx}
\usepackage{hyperref}
\usepackage{float}
\usepackage{epstopdf}
\usepackage{color}
\usepackage{bm}
\usepackage{comment}
\usepackage{soul}
\usepackage{esint}
\usepackage[shortlabels]{enumitem}
\usepackage{eufrak}

\allowdisplaybreaks
\newtheorem{theorem}{Theorem}[section]
\newtheorem{proposition}[theorem]{Proposition}
\newtheorem{lemma}[theorem]{Lemma}
\newtheorem{corollary}[theorem]{Corollary}
\newtheorem{remark}[theorem]{Remark}

\newtheorem{example}[theorem]{Example}
\newtheorem{definition}[theorem]{Definition}

\def\diam{\mathrm{diam}}

\def\mcE{\mathcal{E}}
\def\mcF{\mathcal{F}}

\def\sE{\mathcal{E}}
\def\sF{\mathcal{F}}
\def\sG{\mathcal{G}}
\def\sX{\mathcal{X}}

\def\sL{\mathcal{L}}

\def\R{{\mathbbm R}}

\def\bP{{\mathbbm P}}
\def\bE{{\mathbbm E}}

 \def\1{{\mathbbm 1}}

 \makeatletter
\@namedef{subjclassname@2020}{%
  \textup{2020} Mathematics Subject Classification}
\makeatother

\def\eps{\varepsilon}
\def\wt{\widetilde}

\def\<{\langle}
\def\>{\rangle}

\numberwithin{equation}{section}

\begin{document}
\title[Uniform boundary Harnack principle for non-local operators]{Uniform boundary Harnack principle for non-local operators on metric measure spaces}
 
\author{Shiping Cao}
\address{Department of Mathematics, University of Washington, Seattle, WA 98195, USA}
\email{spcao@uw.edu}
\thanks{}

\author{Zhen-Qing Chen}
\address{Department of Mathematics, University of Washington, Seattle, WA 98195, USA}\email{zqchen@uw.edu}
\thanks{}

\subjclass[2020]{Primary }

\date{}


\maketitle

\begin{abstract}
We obtain a uniform boundary Harnack principle (BHP) on any open sets for a large class of non-local operators on metric measure spaces under a jump measure comparability and tail estimate condition, and an upper bound condition on the distribution function for the exit times from balls.
 These conditions are satisfied by any non-local operator $\sL$ that admits a two-sided mixed stable-like heat kernel bounds when  the underlying metric measure spaces have  volume doubling and reverse volume doubling properties. The results of this paper are new even for non-local operators on Euclidean spaces.
In particular, our results give not only the scale invariant but also uniform BHP for the first time for  non-local operators on Euclidean spaces  of both divergence form and    non-divergence form  with measurable coefficients. 
\end{abstract}

\section{\bf Introduction}

Let $\sL$ be an operator  on a locally compact metric space $(\sX,d)$ that generates a Hunt process $X$. 
  For each $x\in \sX$ and $r>0$, denote by  $B(x,r):=\{y\in \mathcal{X}:\,d(x,y)<r\}$   the open ball centered as $x$ with radius $r$. We say the boundary Harnack principle (BHP in abbreviation)  for $\sL$  (or equivalently for the Hunt process $X$)   holds  on an open subset $D\subset\sX$ if   for any $\xi\in\partial D$ and $r\in (0, r_0)$ for some $r_0>0$, there is a constant  $C=C(\xi, r)>1$   so that  
\[
h_1(x)h_2(y)\leq C\,h_1(y)h_2(x)\quad\hbox{ for every }x,y\in D\cap B(\xi,r)
\]
for any non-negative functions $h_1$ and $h_2$ that are $\sL$-harmonic in $D\cap B(\xi,2r)$ and vanish on $B(\xi,2r)\setminus D$ in a suitable sense. If there is  some $R \in (0, \infty]$ so that the above holds for some constant $C>1$ uniform in  $\xi \in \partial D$ and $r\in (0, R)$,  we say the scale invariant BHP holds for $X$ and $\sL$ with radius $R$ on $D$.  If, furthermore, $C>1$ can be taken to be independent of the open subset $D\subset \sX$,  we say the uniform BHP holds for $X$ and $\sL$ with radius $R$. Whether BHP holds or not is a fundamental problem both in  analysis and in probability theory.  When it holds, it has many far-reaching implications. For instance,  the scale invariant BHP plays a crucial role in the study of free boundary problems  in PDE, Martin boundary, Fatou theorem and Green function estimates of $\sL$ in open sets.

In this paper, we show that a uniform BHP holds on any open sets for a large class of non-local operators on metric measure spaces under a jump measure comparability and tail estimate condition and an upper bound condition on the distribution function for the exit times from balls.
 These conditions are satisfied by any non-local operator $\sL$ that admits a two-sided mixed stable-like heat kernel bounds when  the underlying metric measure spaces have  volume doubling and reverse volume doubling properties. 
 The non-local operators studied in this paper can be   symmetric as well as  non-symmetric.  
 See Theorem \ref{T:1.3}, Corollary \ref{C:1.5} and Corollary \ref{C:1.7} below for precise statements. 
 These results  are new even for non-local operators on Euclidean spaces.
In particular, our results give not only the scale invariant but also uniform BHP for the first time for  non-local operators on Euclidean spaces  of both divergence form and    non-divergence form  with measurable coefficients. 
\medskip
  
The study of BHP has a long history. 
For Brownian motion and Laplacian on $\R^d$,  it is easy to see that a scale invariant BHP holds on  smooth $C^2$ domains
from  the two-sided explicit estimates of their Poisson kernels. 
For Laplace operators  in Lipschitz domains, 
 the scale invariant BHP was   proved  independently in 1977-1978   by  Dahlberg \cite{D},  Ancona \cite{Ac} and   Wu \cite{Wu}. 
 This   result    is    later extended to  more general domains such as non-tangential accessible (NTA) domains and uniform domains and to more general  elliptic operators in Lipschitz domains
in \cite{JK, CFMS, FGMS, Ai}.  Bass and Burdzy \cite{BB1} developed a probabilistic approach  to give an alternative proof of  BHP in Lipschitz domains
for Brownian motion.
This probabilistic method was further developed   to obtain a non-scale invariant BHP for uniformly elliptic operators   in H\"older domains and twisted H\"older domains in \cite{BB2, BBB, BB4}.
 There is  now
 a very substantial literature on BHP for differential operators.
 For diffusions on geodesic or length spaces,  the scale invariant BHPs
   on inner uniform domains have been established in \cite{GSC, LS, L} under a  two-sided  
 Gaussian or sub-Gaussian heat kernel bounds assumption. For symmetric diffusions on complete metric doubling spaces, 
 the scale invariant elliptic Harnack principle is equivalent to sub-Gaussian heat kernel estimates by a quasi-symmetry of metric and a time change as shown recently in \cite{BM2,BCM}.  Thus for symmetric diffusions,  the scale invariant BHP holds on inner uniform domains of length spaces  \cite{BM, BCM}
 and on uniform domains of complete metric doubling spaces \cite{aC}
 under the assumption of a scale invariant elliptic Harnack inequality. 
 
  Discontinuous Markov processes and non-local operators are important subjects in probability theory and analysis.
 The study of BHP for non-local operators (or equivalently, discontinuous Markov processes) started relatively  recently. 
  Bogdan \cite{Bo} showed that a  scale invariant BHP holds for $\Delta^{\alpha/2}:=-(-\Delta)^{\alpha/2}$
(or equivalently, for isotropic $\alpha$-stable processes) on $\R^d$ 
 with $\alpha\in (0,2)$ in bounded Lipschitz domains.
This result is  extended to  corkscrew open sets in \cite{SW}. 
It is later shown in  \cite{BKK1} that the uniform BHP holds  for  $\Delta^{\alpha/2}$ on $\R^d$ for arbitrary open sets. The uniform BHP is shown to hold in  arbitrary open sets in \cite{KSV2} for a class of rotationally symmetric pure jump symmetric L\'evy processes  whose jump kernels are comparable to that of subordinate Brownian motions, and in \cite{KM} for a class of subordinate Brownian motion including geometric stable processes. The scale invariant  BHP has also been established  on some  class of open sets
 for some  non-homogeneous non-local operators on $\R^d$ that are generators
of discontinuous processes; see \cite{BBC, KS, CRY}. 
A scale invariant BHP has been show to hold on open sets for $s$-stable-subordination with $s\in (0, 1)$ of symmetric diffusion processes that  have
two-sided sub-Gaussian heat kernel estimates on  Sierpinski carpets and gaskets  in \cite{Sto, KK}  
and on Ahlfors $n$-regular metric measure spaces   in \cite[Example 5.7]{BKK}. 
 A non scale invariant BHP is established on any open sets 
 for isotropic-stable-like operators in non-divergence form is given in \cite{RS}, where the harmonic functions are interpreted
 in the  viscosity sense for  extreme operators.

 The scale invariant BHP   is  also shown to hold on  $C^{1,1}$-smooth open sets for 
subordinate Brownian motion with Gaussian components in \cite{CKSV,  KSV3},
 and for a large class of diffusions with jumps in \cite{CW1}, which may not be L\'evy processes. In a  recent paper by 
 Chen and Wang  \cite{CW2},  a necessary and sufficient condition is obtained for the scale invariant BHP for a large class of discontinuous Hunt processes on metric measure spaces that are in weak duality with another Hunt process. In particular, they  showed that the scale invariant BHP   for subordinate Brownian motion with Gaussian components on $\R^d$   holds on Lipschitz domains satisfying an interior cone condition with common angle $\theta \in (\cos^{-1}(1/\sqrt{d}),\pi)$ but fails on truncated circular cones with angle $\theta \in (0, \cos^{-1}(1/\sqrt{d})]$. 
We mention that a non scale invariant BHP  for discontinuous Feller
processes having strong Feller property  (which can possibly have diffusive components)  in weak duality with another Feller process 
on metric measure  spaces is obtained in Bogdan,  Kumagai and  Kwasnicki \cite{BKK}, under a comparability condition  of
the jump kernel and a Urysohn-type  cuf-off function condition on the domains of the generator of the processes and their duals. See also \cite{H} for  an equivalence relation between  uniform BHP  on  corkscrew open subsets of a compact metric space $(\sX, d)$ 
 and  the generalized triangle property of Green functions,  
 under the assumption of certain properties of the associated Green functions. 

 \smallskip 

  In summary, for discontinuous Markov processes (or equivalently, for non-local operators), up to now, the scale invariant BHPs were  only known to hold on arbitrary open sets for a class of rotationally symmetric pure jump L\'evy processes on $\R^d$ and for $s$-stable-subordination of  symmetric diffusion processes on metric measure spaces. For instance,  
 it was not known  until now if  the scale invariant BHP holds or not on general open sets even for isotropic-stable-like non-local operators of divergence form on Euclidean spaces. The purpose of this paper is to establish not only the scale invariant but also uniform BHP on arbitrary open sets for a large class of non-local operators, symmetric or non-symmetric, on metric measure  spaces.  
 We believe the results of this paper will have many important consequences and applications,
 for example in the study of  free boundary problems, Dirichlet and Green function estimates and identification of Martin boundaries
  for non-local operators. 
 
\smallskip 
  
Now, we state the settings and our main results in detail. 

\smallskip 

Let $X =(X_t,\zeta,\sG_t,\bP_x)$ be a Hunt process on $(\sX, d)$, where  $\{\sG_t; t\geq 0\}$ is the minimal admissible augmented filtration generated by the process $\{X_t; t\geq 0\}$, $\bP_x$ is the distribution of the process starting from $x\in\sX$. The random variable $\zeta\in (0,\infty]$ is the lifetime of the process $X$ so that $X_t=\partial$ for $t\geq\zeta$, where $\partial$ is a cemetery added to $\sX$ as a one-point compactification.  It is well known that every Hunt process has a L\'evy system $(J(x, dy), H)$ that describes how the process jumps, where for every $x\in \sX$,
$J(x, dy)$ is a $\sigma$-finite measure on $\sX \setminus \{x\}$ and $H=\{H_t; t\geq 0\}$ is a positive continuous additive functional of $X$.
That is, for any nonnegative function $f$ on $[0, \infty)\times\sX\times\sX$ vanishing along the diagonal of $\sX\times\sX$ and any   stopping time $T$ with respect to $\{\sG_t; t\geq 0\}$, 
\begin{equation}\label{Levysys}
\bE_x \sum_{s\leq T}f(s,X_{s-},X_s)=\bE_x\int_0^T\int_\sX f(s,X_s,y)J(X_s,dy)dH_s.
\end{equation}
 Denote by $\mu_H$ the Revuz measure of the positive continuous additive functional $H$ of $X$.
Note that L\'evy system $(J(x, dy), H)$  for a Hunt process is not unique but $J(x, dy)\mu_H(dx)$ is unique. If  $(J(x, dy), H)$ is a  L\'evy system  for $X$, then so does $(\gamma (x)J(x, dy), \widetilde H)$ for any 
  strictly positive function $\gamma$ on $\sX$, where $\widetilde H_t:=\int_0^t \gamma (X_s)^{-1} dH_s$. 
For a Borel measurable set $D\subset \sX$, we denote by $\tau_D$ the first exit time from $D$ by $X$: 
$$
\tau_D=\inf\{t\geq 0:\,X_t\notin D\}.
$$
 The L\'evy system formula \eqref{Levysys}
  in particular implies that for any proper open subset $D\subset \sX$ and any bounded $\varphi \geq 0$ on $D^c$,
\begin{equation}\label{e:exit}
\bE_x\left[ \varphi (X_{\tau_D}); X_{\tau_D} \not=X_{\tau_D -} \right]  =\bE_x\int_0^{\tau_D}  \int_{D^c} \varphi (y) J(X_s,dy)dH_s.
\end{equation}
Denote by $\sL$ the infinitesimal generator of the Hunt process $X$.
 
 Throughout this paper,  $\phi:[0,\infty)\to [0,\infty)$ is  a strictly increasing function such that $\phi(0)=0$, $\phi(1)=1$ and  there exists $C\in(1,\infty)$ so that 
\begin{equation}\label{phiD}
\phi(2r)\leq C\,\phi(r)\ \hbox{ for all }r>0. 
\end{equation}
Note that \eqref{phiD} is equivalent to the existence of positive constants $c \geq 1$ and $\beta >0$ so that 
\begin{equation}\label{eqnphi}
\frac{\phi(R)}{\phi(r)}\leq c\Big(\frac{R}{r} \Big)^{\beta} 
\quad \hbox{ for all }0<r\leq R<\infty.
\end{equation}

In this paper, we always assume that there is a kernel $J(x,dy)$ such that $(J(x,dy),t)$ is a L\'evy system.  We introduce the following 
 jump measure comparability condition  {\bf (Jc)} and 
 jump measure tail estimates condition {\bf (Jt)}.

\begin{definition}\label{conditionJ}\rm
 \begin{enumerate} [(i)]

\item   We say condition $({\bf Jc})_{\bar{r}}$ holds for $\bar{r}\in (0, \infty]$ if there are constants $C_1\in(0,\infty),C_2\in(0,1),C_3\in(1,\infty)$ and $\theta \in(0,\infty)$ so that for any $x, y\in \sX$, 
\begin{equation}\tag{\bf Jc.1}\label{e:Jc1}
J(x,dz)\leq C_1 \left(1+\frac{d(y,z)}{d(x,z)} \right)^\theta J(y,dz)\quad
\hbox{on }\sX  \ \hbox{ if  }  d(x,y)<\bar{r}, 
\end{equation}
and for $r,s\in(0,\bar{r})$, 
\begin{equation}\tag{\bf Jc.2}\label{e:Jc2}
J(x,B(y,s))\leq C_2\,J(x,B(x,r)^c)\quad \hbox{if }d(x,y)>s+C_3r.
\end{equation}

\item When $\bar{r}=\infty$, we write $({\bf Jc})$ for $({\bf Jc})_\infty$.

\item  We say condition ${\rm \bf (Jt)}_{\phi, \bar r}$ holds for $\bar{r}\in (0,\infty]$ if there are constants $0<C_4\leq C_5$ so that 
\begin{equation}\label{e:Jt} 
\frac{C_4}{\phi(r)}\leq J(x,B(x,r)^c)\leq \frac{C_5}{\phi(r)}
\quad\hbox{ for every }x\in\sX,\,0<r<\bar{r}.
\end{equation}

\item  We say that ${\rm \bf (Jt)}_{\phi, \bar r, \leq }$ (resp. ${\rm \bf (Jt)}_{\phi, \bar r, \geq }$) holds for some $\bar r\in (0, \infty]$ if the upper bound (resp. lower bound) in \eqref{e:Jt} holds.

\item When $\bar r=\infty$, we write ${\rm \bf (Jt)}_{\phi}$ (resp. ${\rm \bf (Jt)}_{\phi,  \leq }$ and ${\rm \bf (Jt)}_{\phi, \geq }$) for  ${\rm \bf (Jt)}_{\phi,\infty}$ (resp. ${\rm \bf (Jt)}_{\phi,\infty, \leq }$ and  ${\rm \bf (Jt)}_{\phi,\infty,\geq}$).

\end{enumerate}
\end{definition}

\begin{remark} \rm 
When $\phi$ is reverse doubling, i.e. there are $C_1,C_2>1$ such that  
\begin{equation}\label{phiRD} 
\phi(C_1r)\geq C_2\phi(r)\quad\hbox{ for every }r>0,
\end{equation}
condition   $({\bf Jt})_{\phi,\bar{r}}$ implies \eqref{e:Jc2}. Indeed, let 
$0<C_4\leq C_5$ be the constants in $({\bf Jt})_{\phi,\bar{r}}$, and   $C_3 :=C_1^k$ with $k= \big[ \frac{ \log (2C_5/ C_4)}{\log C_2} \big]+1$. 
Here for $a\in \R$,  $[a]$ denotes the largest integer not exceeding $a$.
Then, for every  $0< r,s<\bar{r}/C_3$ and $x,y\in\sX$ with $d(x,y)>s+C_3r$, we have by \eqref{phiRD} and $({\bf Jt})_{\phi,\bar{r}}$, 
\begin{align*}
J(x,B(y,s))\leq J(x,B(x,C_3r)^c)\leq\frac{C_5}{\phi(C_3r)}\leq \frac{C_5}{C_2^k\phi(r)}\leq \frac{C_4}{2\phi(r)}\leq \frac{1}{2}J(x,B(x,r)^c). 
\end{align*}
Geometric $\alpha$-stable processes on $\R^d$ with $d\geq 1$ and $\alpha \in (0, 2]$
satisfy  condition  ${\rm \bf (Jt)}_{\phi}$ and  $({\bf Jc})$ with $\phi(r)=\frac1{1-\log(r\wedge1)}\wedge r^\alpha$,
which does not have the reverse doubling property;  see Example \ref{E:geostable}. 
\qed
\end{remark}

We postpone further comments on conditions ${\rm \bf (Jt)}_{\phi, \bar r}$ and {\bf (Jc)$_{ \bar{r}}$}    to Remark \ref{R:1.8}.

\begin{definition} \rm
We say ${\bf (EP)} _{\phi, r_0,\leq}$ holds for $r_0\in (0, \infty]$ if there is a constant  $C>0$  such that 
\begin{equation}\label{e:ep}
\bP_x\big(\tau_{B(x,r)}<t\big)\leq  C  \frac{ t}{\phi(r)}\quad\hbox{ for every } t>0, \,  x\in \sX \hbox{ and } r\in (0,r_0).   
\end{equation}
When $r_0=\diam(\sX)$, we denote ${\bf (EP)} _{ \phi,\diam(\sX), \leq} $ by ${\bf (EP)} _{\phi,\leq}$. 
\end{definition}

\begin{remark} \rm 
Clearly, one can replace the strict inequality $<t$ in \eqref{e:ep} by $\leq t$. 
Moreover, condition ${\bf (EP)} _{\phi, r_0,   \leq}$ is equivalent to the following seemingly stronger condition:
there is a constant  $\tilde C>0$  so that for every $ t>0$,   $x\in \sX$  and   $r\in (0,r_0)$, 
\begin{equation}\label{e:1.7a}
\bP_y\big(\tau_{B(x,r)}<t\big)  \leq     \frac{\tilde C  t}{\phi(r )} \quad \hbox{for every } y\in B(x, r/2).
\end{equation}
This is because  for any $x\in \sX$,  $t>0$ and $r\in (0, r_0)$,
we have by     ${\bf (EP)} _{\phi, r_0,\leq}$ and  \eqref{eqnphi},
$$
\bP_y\big(\tau_{B(x,r)}<t\big) \leq \bP_y\big(\tau_{B(y, r/2)}<t\big) \leq  C  \frac{ t}{\phi(r/2)}\leq      \frac{ \tilde C  t}{\phi(r)},
$$
where $\tilde C= c_2 2^{\beta}  C$.  
 Condition \eqref{e:1.7a} and hence condition ${\bf (EP)} _{\phi, r_0,   \leq}$  can be  stated analytically as follows.
 There is a constant  $\tilde C>0$  so that 
  for every $ t>0$,   $x\in \sX$  and   $r\in (0,r_0)$, 
\begin{equation}\label{e:1.8}
 P^{B(x, r)}_t 1 (y) \geq 1-   \frac{\tilde C  t}{\phi(r)}\quad\hbox{ for   } y\in B(x, r/2).
\end{equation}
Here  $\{P^{B(x, r)}_t; t>0\}$ is the Dirichlet heat semigroup of $\sL$ in $B(x, r)$, or equivalently, the transition semigroup of  the part process $X^{B(x, r)}$ of $X$ killed  upon exiting $B(x, r)$. 
\qed 
   \end{remark}

\medskip

\begin{definition} \rm 
\begin{enumerate} [(i)]
 \item A function $h$ on $\sX$ is said to be {\it regular harmonic}  with respect to $X$ (or equivalently, with respect to $\sL$) 
 in an open subset $U\subset \sX$ if 
$ \bE_x[|h(X_{\tau_U})|;\tau_U<\infty]<\infty$  and $ h(x)=\bE_x[h(X_{\tau_U}); \tau_U<\infty ] $  for every $ x\in U$.

\item  A function $h$ on $\sX$ is said to be {\it   harmonic} with respect to $X$ (or equivalently, with respect to $\sL$)  
in an open subset $U\subset \sX$ if for every relatively compact open subset $D$ of $U$,
$h$ is regular harmonic  with respect to $X$    in $D$.
\end{enumerate} 
\end{definition} 
\medskip

The above definition of harmonicity is probabilistic. We refer the reader to \cite{C} for the equivalence between the probabilistic and analytic notions of harmonicity when the process $X$ is symmetric, and \cite{MZZ} when $X$ is non-symmetric but its infinitesimal generator satisfies a sectorial condition. 

The following are the main results of this paper.

\begin{theorem}\label{T:1.3}
Suppose that   $({\bf Jc})_{\bar{r}}$,  ${\rm \bf (Jt)}_{\phi,\bar r}$ and ${\bf(EP)}_{\phi,r_0,\leq}$ hold for some $\bar{r},r_0\in(0,\infty]$.  Then the uniform boundary Harnack principle holds for $\sL$ with radius $R:= \min\{\bar{r}/2, r_0\}$ on any proper open set $D\subset \sX$ in the following sense.  For  every $\kappa\in (0,2)$, there is a constant $C \in (1,\infty)$ that depends only on $\kappa$ and the parameters in 
$({\bf Jc})_{\bar{r}}$,  ${\rm \bf (Jt)}_{\phi, \bar r}$, ${\bf (EP)} _{\phi, r_0, \leq}$ and $\eqref{eqnphi}$ such that for any proper open set $D\subset \sX$, $r\in (0, R)$, $\xi\in \partial D$ and any two non-negative functions $h_1$ and $h_2$ that are regular harmonic with respect to $X$ in $D\cap B(\xi,2r)$ and vanish on $B(\xi, 2r)\setminus D$, we have
\[
h_1(x)h_2(y)\leq C\,h_1(y)h_2(x)
\quad \hbox{ for any }x,y\in D\cap B(\xi,\kappa r).
\]
\end{theorem}
\medskip 

The above type BHP is called uniform because the comparison constant $C>0$ not only is independent of $\xi \in \sX$ and $r>0$ but also is independent of the open set $D$. This kind of uniform BHP is first obtained for isotropic stable processes in $\R^d$ in \cite{BKK1}.

\smallskip

Let $m$ be a Radon measure on a locally compact metric space $(\sX,d)$ with full support. 
We say  the metric measure space $(\sX, d, m)$
is volume doubling (VD) if there is $C_1\in (1,\infty)$ such that
$$
V(x,2r) \leq C_1 \, V(x,r) \qquad \hbox{ for every }  x\in \sX   \hbox{ and }   r\in (0,\infty),
$$
where $V(x, r):= m(B(x, r))$. 
This is equivalent to the existence of positive constants $c_1$ and $d_1$ so that 
\begin{equation}\label{e:vd}
\frac{ V(x, R)}{V(x, r)} \leq c_1  \Big(\frac{R}{r} \Big)^{d_1}  \quad 
\hbox{ for every }x\in \sX   \hbox{ and }   0<r\leq R<\infty.
\end{equation}
We say that reverse volume doubling property (RVD) holds if there are positive constants $c_2$ and $d_2$ so that 
$$
\frac{V(x,R)} {V(x,r)} \geq  c_2 \Big(\frac{R}{r} \Big)^{d_2}  \quad 
\hbox{ for every }x\in \sX   \hbox{ and }   0<r \leq R  \leq {\rm diam}(\sX).
$$
This is equivalent to the existence of $\lambda_0\geq 2 $ and $C_2 >1$ so that 
\begin{equation}\label{e:rvd}
V(x, \lambda_0 r)   \geq C_2  V(x, r) \quad \hbox{ for every }x\in \sX   \hbox{ and }   0<r    \leq   {\rm diam}(\sX)/\lambda_0.
\end{equation}
It is known that VD implies RVD if $\sX$ is connected; See \cite[Proposition 2.1 and  the  paragraph before Remark 2.1]{YZ}.

\begin{definition} \rm 
We say ${\bf J}_{\phi}$ holds if there is a kernel $J(x,y)$ and  positive  constants $C_1,C_2$ such that $(J(x,dy):=J(x,y)m(dy), t)$  is a L\'evy system for $X$ and
\[
\frac{C_1}{V(x,d(x,y)) \phi\big(d(x,y)\big)}\leq J(x,y)\leq \frac{C_2}{V (x,d(x,y) )\phi\big(d(x,y)\big)}\quad\hbox{ for every }x\not= y\in \sX.
\]
We say that ${\bf J}_{\phi, \leq}$ (resp. ${\bf J}_{\phi, \geq}$) holds if the upper bound (resp. lower bound) in the above formula holds. \medskip 
\end{definition} 

 Suppose $(\sX,d,m)$ is VD and RVD, and $\phi$ is reverse doubling.
We will show in Lemma \ref{lemma22}  that ${\bf J}_\phi$ implies that 
  {\bf (Jc)} and ${\rm \bf (Jt)}_{\phi,\bar r}$  hold with $\bar{r}=\diam(\sX)/\lambda_0$,  
  where $\lambda_0\geq 2 $ is the constant in {\rm RVD} \eqref{e:rvd}. 
Thus we have the following from  Theorem \ref{T:1.3}.

\begin{corollary}\label{C:1.5}
 Suppose  that $(\sX,d,m)$ is {\rm VD}  and {\rm RVD}, and  that the increasing function  $\phi$ is reverse doubling.
 Suppose   ${\bf J}_\phi$ and ${\bf (EP)} _{\phi, r_0,   \leq}$ hold for some $r_0\in (0, \infty]$. 
 Then the uniform boundary Harnack principle holds for $\sL$ with radius $ R:= \frac{\diam(\sX)}{2\lambda_0}\wedge r_0$. That is, for  every $\kappa\in (0,2)$, there is a constant $C \in (1,\infty)$ that depends only on $\kappa$ and the parameters in {\rm VD}, {\rm RVD}, 
 \eqref{eqnphi}, \eqref{phiRD},     ${\bf J}_\phi$ and ${\bf (EP)} _{\phi, r_0,   \leq}$  so  that for any proper open set $D\subset \sX$, $r\in (0, R)$, $\xi\in \partial D$ and any two non-negative functions $h_1$ and $h_2$ that are regular harmonic with respect to $X$ in $D\cap B(\xi,2r)$ and vanish on $B(\xi, 2r)\setminus D$, we have
\[
h_1(x)h_2(y)\leq C\,h_1(y)h_2(x)
\quad \hbox{ for any }x,y\in D\cap B(\xi,\kappa r).
\]
\end{corollary}

\medskip

\begin{definition} \rm
We say the Hunt process $X$ satisfies two-sided heat kernel estimates  ${\bf HK}(\phi, t_0)$ for some $t_0 \in (0, \infty]$ if   $X$ has infinite lifetime and it has a transition density function $p(t,x,y)$ with respect to the measure $m$ on $\sX$ so that  
\begin{align} \label{e:HKE}
c_1 \Big( \frac{1}{V(x,\phi^{-1}(t))}\wedge & \frac{t}{V(x,d(x,y))\phi(d(x,y))}  \Big)  \nonumber
\\& \, \leq p(t,x,y)  \, \leq \, c_2 \Big(  \frac{1}{V (x,\phi^{-1}(t))}\wedge\frac{t}{V(x,d(x,y))\phi(d(x,y))} \Big)
\end{align}
for $x,y\in \sX$ and $t\in(0,t_0)$. When $t_0=\infty$, we simply denote ${\bf HK}(\phi,\infty)$ by ${\bf HK}(\phi)$. 
\end{definition}

\medskip

If the lower bound in \eqref{e:HKE} holds for all $t>0$, then the Hunt process $X$ has infinite lifetime by 
\cite[Proposition 3.1(2)]{CKW}. 
We refer the reader to \cite{CK, CK2, CZ1, CZ2} as well as Section \ref{S:4} 
for examples of  symmetric as well as  non-symmetric Hunt processes   for which ${\bf HK}(\phi, t_0)$ holds. 
See also \cite[Theorem 1.13]{CKW} and \cite[Remark 8.3]{CC} 
for  equivalent characterizations of ${\bf HK}(\phi)$
when $X$ is a symmetric Hunt process. 

\medskip 

Note that by VD \eqref{e:vd}, for any $x, y \in \sX$ and $t>0$, 
\begin{align*}
&\frac{V(y, \phi^{-1}(t))}{V(x, \phi^{-1}(t))}  \leq \frac{V(x, d(x, y)+\phi^{-1}(t))}{V(x, \phi^{-1}(t))}
\leq c_3 \left(1 + \frac{d(x, y))}{\phi^{-1}(t))}\right)^{d_1},\\
& \frac{V(y,d(x,y))}{V(x,d(x,y))}\leq \frac{V(x,2d(x,y))}{V(x,d(x,y))}
\leq c_32^{d_1}.
\end{align*}
So for any $x, y\in \sX$ and $t>0$, 
\begin{equation}\label{e:1.3} 
  \frac{1}{V (x,\phi^{-1}(t))}\wedge\frac{t}{V(x,d(x,y))\phi(d(x,y))}  
\leq   c_32^{d_1} \Big(  \frac{1}{V (y,\phi^{-1}(t))}\wedge\frac{t}{V(y,d(x,y))\phi(d(x,y))} \Big) .
 \end{equation}
Hence under condition ${\bf HK}(\phi, t_0)$, there is a constant $c\geq 1$ so that 
\begin{equation}\label{e:1.4}
p(t, x, y) \leq c \, p(t, y, x) \quad \hbox{for every } x, y\in \sX \hbox{ and } t\in (0, t_0).
\end{equation}

\medskip

It will be shown in Proposition \ref{P:2.1} that 
${\rm\bf HK}(\phi, t_0)$ implies both ${\bf J}_{\phi}$ and $ {\bf (EP)}_{\phi, r_0,   \leq}$ for every $r_0>0$. 
Consequently, we have the following. 

\medskip

\begin{corollary}\label{C:1.7}
  Suppose that the increasing function $\phi$ is reverse doubling.
Suppose that   $(\sX, d, m)$ satisfies {\rm VD}  and {\rm RVD},  and  ${\bf HK}(\phi, t_0)$ holds for some $t_0\in (0, \infty]$.
 The uniform BHP holds for $\sL$ with radius $R$ on any proper open set $D\subset \sX$, where  $R:=\frac{\dim(\sX)}{2\lambda_0}$ if $\diam(\sX)<\infty$, $R$ can be any positive number if $t_0<\infty$ and $\diam(\sX)=\infty$,  and $R:=\infty$ when $t_0={\rm diam} (\sX)= \infty$,   where  $\lambda_0\geq 2 $ is the constant in {\rm RVD} \eqref{e:rvd}.
 More precisely, for every $\kappa\in (0,2)$, there is a constant $C\in (1,\infty)$ that depends only on $\kappa$, $\phi(R)/t_0$ (when $t_0<\infty$) and the bounds in {\rm VD}, {\rm RVD},  \eqref{eqnphi}, \eqref{phiRD}  and  ${\bf HK}(\phi, t_0)$    such that for any proper open set $D\subset \sX$,  $r\in (0, R)$, $\xi\in \partial D$ and any two non-negative functions $h_1$ and $h_2$ that are regular harmonic  in $D\cap B(\xi,2r)$ with respect to the process $X$
and vanish on $B(\xi, 2r)\setminus D$, we have
\[
h_1(x)h_2(y)\leq C\,h_1(y)h_2(x)\quad \hbox{ for any }x,y\in D\cap B(\xi,\kappa r).
\]
\end{corollary}

\medskip

The proof of  the main result, Theorem \ref{T:1.3}, will be given in Section \ref{S:3.1}.  
The approach of this paper is purely probabilistic. It is quite robust, and works for a large class of non-local operators including those 
with very light intensity for small jumps such as the generators of  geometric stable processes.
 We do not need estimates of potential kernels and Poisson kernels of the Hunt process. In particular, our method applies to non-symmetric non-local operators, and we do not require the existence of a dual process.  See Examples \ref{E:4.6} and \ref{E:4.7} 
in Section \ref{S:4} for an illustration.

\medskip

We now make some comments about the  conditions  ${\bf (Jc)}_{\bar{r}}$  and ${\rm \bf (Jt)}_{\phi, \bar r}$.

\begin{remark}\label{R:1.8}   \rm  
 \begin{enumerate} [(i)]
 \item  Clearly, if $({\bf Jc})_{\bar{r}}$ holds for some $\bar{r}\in(0,\infty]$, then ${\bf (Jc)}_{\bar{s}}$ holds for every 
$\bar{s} \in (0, \bar{r})$.  The same is true for   condition   ${\bf (Jt)}_{\phi, \bar r}$.

 \item Suppose  that $m$ is a measure on $\sX$ and $\psi$ is an increasing positive function on $(0, \infty)$ that has the doubling property,
that is, there is some $c>0$ so that $\psi (2r) \leq c\psi (r)$ for every $r\in (0, \infty)$. 
The latter is equivalent to that there are positive constants $c\geq 1$ and $\theta>0$ so that 
$$ 
\frac{\psi (R)}{\psi (r) } \leq c (R/r)^\theta  \quad \hbox{  for any } 0<r<R.
$$ 
If $J(x, z)$ is a positive function on $\sX\times \sX \setminus {\rm diagonal}$  so that 
\begin{equation}\label{Jump}
\frac{c_1}{\psi( d(x,z) )} \leq J(x, z) \leq \frac{c_2} {\psi (  d(x,z) )}  \quad \hbox{for } x\not= z \in \sX
\end{equation}
for some constants $c_2>c_1>0$.  Then $J(x, dz) =J(x,z)m(dz)$ satisfies  \eqref{e:Jc1}. 

  If in addition $(\sX,d,m)$ has the property that there is $c_3\in (0, 1)$ such that
\begin{equation}\label{geo}
m\big(B(y,s)\cap (B(x,R)\setminus B(x,r))\big)\leq c_3m(B(x,R)\setminus m(x,r))
\end{equation}
for  any $R>r>0$, $s>0$  and $x,y\in\sX$ with $d(x,y)>r+s$, then by \eqref{eqnphi}
 $J(x, dz) =J(x,z)m(dz)$ has the property that 
\begin{eqnarray}
&&\ J\big(x,\big(B(x,r)\cup B(y,s)\big)^c\big)  \nonumber \\
&=&\sum_{k=0}^\infty J\big(x, B(y,s)^c\cap(B(x,2^{k+1}r)\setminus B(x,2^kr)) \big)  \nonumber  \\
&\geq & \sum_{k=0}^\infty \frac{c_1}{\psi(2^{k+1}r)} m\big(B(y,s)^c\cap(B(x,2^{k+1}r)\setminus B(x,2^kr)) \big) \nonumber  \\
&\geq& \sum_{k=0}^\infty  \frac{ c_1}{ c 2^{\theta}  \psi(2^kr)} \cdot  (1-c_3)\,m\big(  B(x,2^{k+1}r)\setminus B(x,2^kr)  \big)  \nonumber \\
&\geq & \sum_{k=0}^\infty \frac{ c_1  (1-c_3)  }{ c 2^{\theta}  \psi(2^kr)} \, m\big( B(y,s)\cap(B(x,2^{k+1}r)\setminus B(x,2^kr)) \big)
\nonumber \\
&\geq & c_4  \,J\big(x,B(y,s)\big).   \label{e:1.17}
\end{eqnarray}
where $c_4:=  \frac{ c_1  (1-c_3)  }{ c2^\theta  c_2  }$.  
It follows that   
\begin{align*} 
J(x,B(x,r)^c)&=J(x,B(y,s))+J(x,(B(x,r)\cup B(y,s))^c)\\
&\geq (1+c_4)J(x,B(y,s))
\end{align*}
and  so  \eqref{e:Jc2} holds. Consequently,  $({\bf Jc})$ holds.

Note that the Euclidean space $(\R^d,d,dx)$ equipped with Lebesgue measure
satisfies condition \eqref{geo} with $c_3=1/2$.  

\item  More generally, suppose that there are $ \bar{r} \in(0,\infty)$, $c\geq 1 $ and $\theta>0$ so that 
\begin{equation}\label{e:1.7}
\begin{cases}
\frac{\psi (R)}{\psi (r) } \leq c (R/r)^\theta  \quad &\hbox{for any }0<r<R\leq  \bar r, 
     \\
\psi (r+1) \leq c \phi (r)  \quad & \hbox{for any } r\geq  \bar r, 
\end{cases} 
\end{equation} 
and  $J(x,z)$ is a kernel satisfying \eqref{Jump}. Then, $J(x,z)m(dz)$ satisfies  \eqref{e:Jc1}.
Indeed,  for every $ x,  y, z \in\sX$  with $d(x,y)<  \bar r $,    \eqref{e:Jc1}   clearly holds when $d(x, z) \geq d(y, z)$ or 
when $d(y, z)\leq  \bar r $. When  $ \bar r \leq d(x, z) \leq d(y, z)$,  by \eqref{e:1.7}
 \begin{eqnarray*}
\frac{J(x, z)}{J(y, z)} \lesssim \frac{\psi (d(y, z))}{\psi (d(x, z))} \leq 
  \frac{\psi (d(x, y) + d(x , z))}{\psi (d(x, z))} \leq \frac{\psi (  \bar r  + d(x, z))}{\psi (d(x, z))}
    \lesssim 1 .
 \end{eqnarray*}
 When    $d(x, z)<\bar r< d(y, z)$ and $d(x, y) \leq \bar r$, 
$$ \bar r <d(y, z)\leq d(y, x) + d(x, z) \leq 2\bar r
$$
and so by \eqref{e:1.7},
$$ 
\frac{J(x, z)}{J(y, z)} \lesssim  \frac{\psi (d(y, z))}{\psi (d(x, z))}  \leq  \frac{ c^{ [\bar r]+1}\psi (\bar r)}{\psi (d(x, z))}
\lesssim  \left( \frac{\bar r}{d(x, z)} \right)^\theta 
\lesssim  \left( \frac{d(y, z)}{d(x, z)} \right)^\theta . 
$$
  This proves  that $J(x,z)m(dz)$ satisfies  \eqref{e:Jc1}. By the same reasoning as in \eqref{e:1.17}, 
   if \eqref{geo} holds for $(\sX,d,m)$, then $J(x,z)m(dz)$ satisfies $({\bf Jc})_{\bar{r}}$.

Function $\psi (r)=e^{\lambda r^\beta}$ with $\lambda >0$ and $\beta \in (0, 1]$ 
has the property that $\psi (r+1)\leq e^\lambda  \psi (r) $ for every $r>0$. 
Thus an increasing positive function $\psi$ satisfying condition \eqref{e:1.7} could have exponential 
or sub-exponential growth at infinity. 
Thus conditions  ${\bf (Jc)}_{\bar r}$  and ${\bf (Jt)}_{\phi, \bar r}$ with finite $\bar r$ allow exponential and sub-exponential decays in $J(x, z)$ as $d(x, z)\to \infty$.  

\item   The scale invariant BHP may fail if condition ${\bf (Jc)}_r$ is not satisfied even on smooth  Euclidean domains.  
For example, for truncated $\alpha$-stable L\'evy processes on $\R^d$ where $J(x, dy)=\frac{c}{|x-y|^{d+\alpha}} \1_{\{|x-y|\leq 1\}}  dy $, 
it is shown in \cite[Section 6]{KS}    that even non-scale invariant BHP can fail on some smooth domains in $\R^d$. 
It is mentioned  at the end of Example 5.14 in \cite{BKK} but without details that   the  scale invariant BHP can fail on some open set
for L\'evy processes having  L\'evy measure $\nu (dz)=\frac{c}{|z|^{d+\alpha} e^{|z|^\beta}} dz$ with $0<\alpha<2$ and $\beta >1$. 
\qed
\end{enumerate} 
\end{remark}
   
\smallskip

We point out that {\bf HK}$(\phi,t_0)$ implies that $X$ is of pure jump type if $p(t,x,y)$ is symmetric, that is, 
if $p(t,x,y)=p(t,y,x)$ for each $t>0$ and $x,y\in\sX$. Here, pure jump type means that the associated Dirichlet form has no  local part nor killing part. See Lemma \ref{lemmaj1} and Corollary \ref{coroj2} for a proof and the precise statement. 
   When $p(t,x,y)$ is non-symmetric, we do not have such an assertion.

\smallskip

To illustrate  the wide scopes of the main results of this paper,  several concrete examples are given in   Section \ref{S:4} for which the uniform BHP is shown to hold on  any open set of the state space. Here is a partial list of  these examples. 

\begin{enumerate} [(i)]
\item (Example \ref{E:4.4}(i)) Uniform BHP holds on any open sets for any  stable-like non-local operators of divergence form 
on $\R^d$   with measurable coefficients that are bounded between  two positive constants. In fact, it holds  for any   mixed stable-like 
non-local operators of divergence form  on $\R^d$ with uniformly elliptic and bounded measurable coefficients.

\item (Example \ref{E:4.4}(ii)) Uniform BHP holds on any open sets for  a large class of isotropic unimodal L\'evy processes on $\R^d$. 

\item (Example \ref{E:4.4}(iii)) Let $D$ be an open  Ahlfors $d$-regular open subset of $\R^d$. Uniform BHP holds on any open sets  of $\overline D$  for symmetric reflected stable-like processes on $\overline D$. 
In particular, the scale invariant BHP holds for the open set $D$ itself. This   extends significantly the BHP obtained in \cite[Theorem 1.2]{BBC}  for censored stable processes where 
the open set $D$ is  assumed to be $C^{1,1}$-smooth. 
  
\item (Example \ref{E:4.7}) Uniform BHP holds on any open sets for   stable-like non-local   operators  of non-divergence form 
with bounded measurable coefficients. 
 This example   significantly extends the non-scale-invariant BHP result in \cite{RS}.
 
 \item (Example \ref{E:4.8}) Uniform BHP in fact holds on any open sets for   mixed stable-like non-local operators  of non-divergence form 
with bounded measurable coefficients.

\item (Example \ref{E:geostable}) 
Uniform BHP  holds on any open sets for any geometric $\alpha$-stable-like processes on $\R^d$ with $d\geq 1$ and $\alpha \in (0, 2]$.

\item (Examples \ref{E:4.5} and \ref{E:4.9})
Uniform BHP  holds on any open sets for tempered  mixed stable-like non-local operators with bounded measurable coefficients, where the jump kernel decays exponentially or sub-exponentially at infinity.

\item (Example \ref{E:4.10})
Uniform BHP holds on any open sets for SDEs driven by a large class of isotropic L\'evy processes including isotropic mixed stable processes. 
\end{enumerate} 
  
\medskip

The rest of the paper is organized as follows. In Section \ref{S:2}, we discuss various exit time estimates and consequences of heat kernel estimates ${\rm\bf HK}(\phi, t_0)$.  Section \ref{S:3} is devoted to the proof of the main result of this paper, Theorem \ref{T:1.3}. 
Our approach is probabilistic. A key exit probability estimate is given in Subsection \ref{S:3.1}, whose proof uses a box method argument. 
In Section \ref{S:4}, we first show that for a symmetric Markov process $X$, heat kernel estimate  {\bf HK}$(\phi,t_0)$ implies that 
the Dirichlet form for $X$ is regular and of pure jump   type, and that its jump measure has a density kernel.  Several concrete families of discontinuous Markov processes  are    then given for which the uniform BHP holds on arbitrary open sets.

\smallskip

Notations.  In this paper, we use $:=$ as a way of definition. For $a, b\in \R$,  $a\vee b:= \max\{a, b\}$, $a\wedge b := a \wedge b$, $a^+:=a\vee 0$ and $a^-:=(-a)\vee 0$. 
  For $A,  A_1\subset \sX$, $x\in \sX$ and $r>0$, we define 
\begin{align*}
d(A,A_1 )&:=\inf\{d(x,y):x\in A,\,y\in A_1\},\\ 
d(x,A)=d(A,x)&:=\inf\{d(x,y):\,y\in A\},\\
B_A(x,r)&:=B(x,r)\cap A.
\end{align*}
We denote the closure of $A$ in $(\sX,d)$ by $\overline{A}$, and denote the boundary of $A$ in $(\sX,d)$ by $\partial A$.
 Denote by $C(\sX)$ the space of continuous functions on $(\sX, d)$, and $C_c(\sX)$   the space of compactly supported continuous functions on $(\sX,d)$. For $f\in C(\sX)$, $\|f\|_\infty:=\sup_{x\in\sX}|f(x)|$ and $\operatorname{supp}[f]=\overline{\{x\in\sX:\,f(x)\neq 0\}}$. 
Let $C_b (\sX):=\{f\in C(\sX): \| f\|_\infty<\infty\}$.   
 For an open set $D\subset \sX$ and $x\in D$, we denote by  $d_D(x):=\inf\{y\in D^c: d(x, y)\}$  the distance from $x$ to $D^c$.

\medskip

\section{\bf  Implications of heat kernel estimates ${\rm\bf HK}(\phi, t_0)$}\label{S:2} 

 In this section, we consider the relationship between the conditions of Corollaries \ref{C:1.5}, \ref{C:1.7} and Theorem \ref{T:1.3}, and we list some easy consequences of $ {\bf (Jc)}_{\bar r} $,
${\rm \bf (Jt)}_{\phi, \bar r}$ and ${\bf (EP)}_{\phi,r_0,\leq}$.

\medskip 

Throughout this paper, we assume $m$  is a Radon measure on a locally compact metric space $(\sX,d)$ with full support, 
and $\phi$ is a strictly increasing function on $[0, \infty)$  satisfying  \eqref{eqnphi}  with  $\phi(0)=0$ and  $\phi(1)=1$. 
First, we show that  when $\phi$ is reverse doubling,  
\[
{\rm VD}+{\bf HK}(\phi,t_0)\Longrightarrow {\bf J}_\phi+ {\bf (EP)}_{\phi,r_0,\leq},
\]
 where $r_0=\infty$ if $t_0=\infty$, and $r_0$ can be any positive number if $t_0<\infty$. 
   We  recall the following lemma from \cite[Lemma 2.1]{CKW}.

\begin{lemma}\label{L:2.1}
  Suppose that $\phi$ is reverse doubling, and   that $(\sX,d,m)$ is {\rm VD}. 
  There is a constant $c>0$ that depends only on the bounds in \eqref{phiRD}  and \eqref{e:vd}   
   so that for every $x\in \sX$ and $r>0$,
$$
\int_{B(x, r)^c} \frac1{V(x, d(x,y)) \phi (d(x, y))} m(dy) \leq \frac{c}{\phi(r)}.
$$
\end{lemma}

The following is a variant  of  condition ${\bf (EP)}_{\phi,r_0,\leq}$ where the restriction on the radius of balls is replaced by the restriction on time.
\begin{definition}
We say condition ${\bf (EP)} _{\phi,\leq,t_0}$ holds for $t_0\in (0, \infty]$ if there is a constant  $C>0$  such that 
\begin{equation}\label{e:ept0}
\bP_x\big(\tau_{B(x,r)}<t\big)\leq  C  \frac{ t}{\phi(r)}\quad\hbox{ for every } t\in(0,t_0), \,  x\in \sX \hbox{ and } r>0.   
\end{equation}
\end{definition} 

Note that when $t_0=\infty$, condition ${\bf (EP)} _{\phi,\leq,\infty}$ is just ${\bf (EP)} _{\phi,\leq}$.

\begin{proposition} \label{P:2.1}
\begin{enumerate} [\rm (a)]
\item ${\bf (EP)}_{\phi,\leq,t_0}\Longrightarrow {\bf (EP)}_{\phi,r_0,\leq}$ for every $r_0>0$. The constant in ${\bf (EP)}_{\phi,r_0,\leq}$ \eqref{e:ep} can be taken to be $C\vee \frac{\phi(r_0)}{t_0}$, where $C$ is the constant in ${\bf (EP)} _{\phi,\leq,t_0}$ \eqref{e:ept0}. 

\item   Suppose that $\phi$ is reverse doubling. Then
${\rm VD}+{\bf HK}(\phi, t_0)\Longrightarrow {\bf J}_{\phi}+{\bf (EP)}_{\phi,\leq,t_0/2}$.
 \end{enumerate}
\end{proposition} 

\begin{proof}
(a)  Let $r_0>0$.  Since $\phi$ is increasing, we can see that 
\[
\bP_x\big(\tau_{B(x,r)}<t\big)\leq 1\leq \frac{\phi(r_0)}{t_0}\frac{t}{\phi(r)}\quad\hbox{ for every } t\geq t_0, \,  x\in \sX \hbox{ and } r\in (0,r_0).   
\]
Suppose that  ${\bf (EP)}_{\phi,\leq,t_0}$ holds.   The above estimate together with \eqref{e:ept0} implies that 
\[
\bP_x\big(\tau_{B(x,r)}<t\big)\leq \Big( C\vee \frac{\phi(r_0)}{t_0} \Big) 
\frac{t}{\phi(r)}\quad\hbox{ for every } t>0, \,  x\in \sX \hbox{ and } r\in (0,r_0).   
\]
That is, ${\bf (EP)}_{\phi,r_0,\leq}$ holds. 

\medskip

(b) ${\rm\bf HK}(\phi, t_0)\Longrightarrow  {\bf (EP)}_{\phi,\leq,t_0/2}$ holds by the argument as that for  \cite[Lemma 2.7]{CKW}.
We next show ${\rm\bf HK}(\phi, t_0)\Longrightarrow  {\bf J}_{\phi}$.  As mentioned earlier, condition ${\rm\bf HK}(\phi, t_0)$ implies that $X$ is a Feller process having strong Feller property.
Thus  $X$ has a L\'evy system $(J(x, dy), H)$, where $H$ is a positive continuous additive functional of $X$. We divide the proof into three steps. 

  (i) We first show that ${\rm\bf HK}(\phi, t_0)$ implies that for $m$-a.e. $x\in \sX$,    $J(x, dy)$ is abosolutely continuous  with respect to $m(dy)$
and that ${\bf J}_{\phi, \leq}$ holds for $m$-a.e. $x\not= y \in \sX$. Note that it follows from ${\rm\bf HK}(\phi,t_0)$ that 
\begin{equation}\label{e:2.2} 
\lim_{t\to 0} \|P_t \varphi - \varphi \|_{\infty} = 0 \quad \hbox{ for any } \varphi \in C_c(\sX).
\end{equation}   
 Let $f,g\in C_c(\sX)$ be non-negative with $\operatorname{supp}[f] \cap \operatorname{supp}[g] = \emptyset$. 
 Let $U$ be a bounded open neighborhood of $\operatorname{supp}[f]$ such that $g(x)=0$ for each $x\in \overline U$. By the strong Markov property of $X$, for $x\in U$, 
\begin{eqnarray}\label{e:2.3}
 P_tg(x)& :=&
 \bE_x[g(X_t)]  =\bE_x\big[ \big( \bE_{X_{\tau_U}}  g(X_{t-s}) \big) \big|_{s=\tau_U};\tau_U<t\big]  \nonumber \\
&=&\bE_x\big[P_{t-\tau_U}g(X_{\tau_U});X_{\tau_U}\in  \partial U,\tau_U<t\big]    \nonumber \\ 
&& +\bE_x\big[P_{t-\tau_U}g(X_{\tau_U});X_{\tau_U}\in \sX \setminus  \overline{U},\tau_U<t\big].
\end{eqnarray}
Take $r=d({\rm supp}[f],U^c)/2$. By ${\bf (EP)}_{\phi,\leq,t_0/2}$, for $t\in (0, t_0/2)$, 
 \begin{eqnarray}\label{e:2.4}
 && \Big|\frac1t\bE_x\big[P_{t-\tau_U}g(X_{\tau_U});X_{\tau_U}\in \partial U,\tau_U<t\big]\Big|   \nonumber \\
& \leq &  \frac1{t}\bP_x(\tau_U<t)\sup_{ z\in \partial U,\, 0<s\leq t}|P_sg ( z )|  \nonumber \\
&\leq &\frac1{t}\bP_x(\tau_{B(x, r)} <t)\sup_{z \in \partial U,\, 0<s\leq t}|P_sg( z)|   \nonumber \\
&\leq &  \frac{c_1}{\phi(r)}\sup_{ z\in \partial U ,\, 0<s\leq t}|P_s g (z)|, 
 \end{eqnarray}
which, by \eqref{e:2.2},  converges to 0 uniformly in $x\in U$ as $t\to 0$. 
On the other hand, by \eqref{e:2.2} and  ${\bf (EP)} _{\phi,\leq,t_0/2}$,  for $t\in (0, t_0/2)$, 
\begin{eqnarray}\label{e:2.5}
 &&  \frac1t \bE_x\big[ |P_{t-\tau_U}g -g| (X_{\tau_U});X_{\tau_U}\in \sX \setminus  \overline{U},\tau_U<t\big]  \nonumber \\
 &\leq &  \sup_{s\in (0, t]}\|  P_s g -g\|_\infty \frac1t \bP_x ( \tau_U<t)  \nonumber \\
 &\leq &  \sup_{s\in (0, t]}\| P_s g -g\|_\infty \frac1t \bP_x ( \tau_{B(x, r)}<t)  \nonumber \\
  &\leq &  \sup_{s\in (0, t]}\| P_s g -g\|_\infty \frac{c_1}{\phi(r)},
 \end{eqnarray}
which converges to 0 uniformly in $x\in U$ as $t\to 0$. Thus we have by \eqref{e:2.3}-\eqref{e:2.5} that 
\begin{equation}\label{e:2.6a}
\lim_{t\to 0} \left( P_t g(x) - \frac{1}{t} \bE_x\big[ g (X_{\tau_U});X_{\tau_U}\in \sX \setminus  \overline{U},\tau_U<t\big]  \right)=0
\quad \hbox{uniformly on } U.
\end{equation}

Denote by $p^U(t, x, y)$ the transition density function of $X^U$, the subprocess of $X$ killed upon leaving $U$.
That is, 
$$
p^U(t, x, y) = p(t, x, y)- \bE_x \left[ p(t-\tau_U, X_{\tau_U}, y) ;\tau_U<t \right] \quad \hbox{ for } x, y\in U.
$$
In particular we have by ${\bf HK}(\phi,t_0)$  and \eqref{e:1.4} that for $t\in(0,  t_0)$ and $x, y\in U$,   
$$
|p^U(t, x, y)-p(t, x, y)|
\leq  \frac{c_2\,t}{V(y,d(y, U^c))\phi(d(y, U^c))}.
$$
This together with \eqref{e:1.4} and  \eqref{e:2.2} shows that for each $z\in U$, 
\begin{eqnarray}\label{e:2.6}
&& \liminf_{t\to 0} \frac1{t} \int_0^t \int_{{\rm supp}[f]} f(x)  p^U(s, x, z) m(dx)   ds  \nonumber \\
&\geq & \liminf_{t\to 0} \frac1{t} \int_0^t \int_{{\rm supp}[f]} f(x)  \Big(p (s, x, z) - \frac{c_2\,t}{V(z,d(z, U^c))\phi(d(z, U^c))} \Big)m(dx)   ds  
\nonumber \\
&\geq & \liminf_{t\to 0} \frac1{t} \int_0^t \int_{{\rm supp}[f]} f(x)  c_3 p (s, z, x) m(dx) ds \nonumber \\
&=& c_3 f(z). 
\end{eqnarray}

For simplicity, write $\int_{\sX} g(y) J(z, dy)$ by $J g(z)$. 
By the L\'evy system \eqref{Levysys} of $X$, 
\begin{eqnarray}\label{e:2.7}
 &&  \int_\sX f(x) \bE_x\big[ g(X_{\tau_U});X_{\tau_U}\notin \overline{U},\tau_U<t\big]m(dx)  \nonumber \\
&=&   \int_{{\rm supp}[f]} f(x) \bE_x \Big[ \int_0^{\tau_U\wedge t} \int_{\sX \setminus \overline U} 
 g  (y) J(X_s,dy  )dH_s \Big] m(dx)  \nonumber \\
 &=& \int_{{\rm supp}[f]} f(x)  \int_0^t  \int_{U}  p^U(s, x, z)  J g (z ) \mu_H  (dz) \,  ds \,  m(dx)  \\
 &=&   \int_{ U}      \left(  \int_0^t \int_{{\rm supp}[f]} f(x)  p^U(s, x, z) m(dx)   ds\right)     
  J g (z )   \mu_H(dz) .  \nonumber  
  \end{eqnarray}
Combining this with \eqref{e:2.6}, \eqref{e:2.6a}  and ${\bf HK}(\phi, t_0)$, we have by the Fatou's lemma that 
\begin{eqnarray} \label{e:2.8}
&&  \int_{\sX\times \sX } f(z) g(y) J(z, dy) \mu_H  (dz)  \nonumber \\
&\leq& \int_{\sX\times \sX}\Big(\liminf_{t\to 0} \frac1{c_3t} \int_0^t \int_{{\rm supp}[f]} f(x)p^U(s, x, z)m(dx)ds\Big) g(y) J(z, dy) \mu_H  (dz)\nonumber\\  
&\leq & \liminf_{t\to 0} \frac{ 1}{c_3 \, t} \int_\sX f(x) \bE_x\big[ g(X_{\tau_U});X_{\tau_U}\notin \overline{U},\tau_U<t\big]m(dx)  \nonumber \\
&\leq & \liminf_{t\to 0} \frac{ 1}{c_3 \, t}   \int_{\sX} f(x) P_t g(x) m(dx) \nonumber \\
&\leq & \int_{\sX\times \sX}  f(x) g(y) \frac{c_3^{-1} c_4} {V(x,d(x,y))\phi(d(x,y))} m(dx)m(dy) .
\end{eqnarray}
Since the above holds for any non-negative $f, g\in C_c(\sX)$ with ${\rm supp}[f]\cap {\rm supp}[g]=\emptyset$, we conclude that 
\begin{equation}\label{e:2.9a}
J(x, dy)    \mu_H(dx) \leq   \frac{c_3^{-1} c_4} {V(x,d(x,y))\phi(d(x,y))} m(dx)m(dy) \quad \hbox{on } 
\sX\times \sX \setminus {\rm diagonal}. 
\end{equation} 
This in particular implies that $\mu_H(dx)\ll m(dx)$.  Without loss of generality, we can and do assume $H_t=t$, or equivalently, $\mu_H(dx)=m(dx)$.
(Otherwise, letting $\gamma (x):= \frac{\mu_H(dx)}{m(dx)}$, then $(\gamma (x)J(x, dy), t)$ is a L\'evy system of $X$ and define $\gamma (x)J(x, dy)$
to be the new $J(x, dy)$.) 
 It follows from  that \eqref{e:2.9a} there is some $J(x, y)\geq 0$  
 so that  for $m$-a.e. $x\in \sX$,  $J(x, dy)= J(x, y) m(dy)$ on $\sX\setminus \{x\}$   
and 
\begin{equation}\label{e:2.9}
    J(x, y)\leq \frac{  c_3^{-1} c_4} {V(x,d(x,y))\phi(d(x,y))} 
\quad m \hbox{-a.e. on } \sX \setminus \{x\}.
\end{equation} 
This proves that ${\bf J}_{\phi, \leq}$ holds $m$-a.e. $x\not=y$  in  $\sX  $.

\medskip

(ii)    We next  show that ${\bf J}_{\phi, \geq}$ holds  $m$-a.e. $x\not=y$  in  $\sX$.
 Lemma  \ref{L:2.1} and \eqref{e:2.9} in particular implies that  $Jg$ is locally bounded on $\sX \setminus {\rm supp}[g]$. 
As $\overline U\subset \sX \setminus {\rm supp}[g]$, $Jg$ is bounded on $\overline U$. 
 Observe that for any $ h \in C_b(U)$, $P^U_t h (x):=\int_U p^U(t, x, y)h(y)m(dy)= \bE_x[h(X^U_t)]$ 
 converges boundedly to $h$ on $U$ as $t\to 0$, and   that by \eqref{e:1.3} for any $\psi\in L^1(U; m)$ and $t<t_0$,
\begin{eqnarray*}
 \int_U | P^U_t \psi|(x) m(dx)
& \leq & \int_{\sX} \int_{\sX} (|\psi|\1_U)(y)  p(t, x, y) m(dy) m(dx) \\
& \leq & c \int_{\sX} \int_{\sX} (|\psi|\1_U)(y)  p(t, y, x)   m(dx) m(dy) \\
&=& c \int_U |\psi(y)| m(dy).
\end{eqnarray*}
It follows for any $\psi \in L^1(U;  m)$,
$$
\lim_{t\to 0} \int_U |P^U_t \psi (x) -\psi (x)| m(dx)=0.
$$
As $\1_U Jg \in L^1(U;m)$, we have by   \eqref{e:2.6a} and \eqref{e:2.7} that 
\begin{eqnarray}\label{e:aa7}
 &&\lim_{t\to 0}\frac1{t}\int_\sX f(x)P_tg(x)m(dx) \nonumber\\
  &= & \lim_{t\to 0}\frac1{t}  \int_\sX f(x) \bE_x\big[ g(X_{\tau_U});X_{\tau_U}\notin \overline{U},\tau_U<t\big]m(dx)  \nonumber\\
 &=&  \lim_{t\to 0}\frac1{t}  \int_{{\rm supp}[f]} f(x)  \int_0^t  \int_{U}  P^U_s (  J g )(z) m(dz) \,  ds \,  m(dx)  \nonumber\\
 &=& \int_{\sX\times \sX}  f(x) g(y) J(x,dy)m(dx).
\end{eqnarray}
We then deduce from the lower bound in ${\bf HK}(\phi, t_0)$ that there is a constant $c_5>0$ so that 
$m$-a.e. $x\in \sX$,   
\begin{equation}\label{e:2.11}
   J(x, y)\geq \frac{c_5} {V(x,d(x,y))\phi(d(x,y))}  \quad m \hbox{-a.e. on } \sX \setminus \{x\}.
\end{equation} 

\smallskip

(iii) Note that by definition \eqref{Levysys} of a L\'evy system for $X$, 
under the condition ${\rm\bf HK}(\phi)$ (in fact,  it is sufficient that $X$ has a density function $p(t, x, y)$ with respect to $m$),
if $(J(x, dy), t)$ is a L\'evy system for $X$ and $J_1 (x, dy)$ is a kernel so that $J(x, dy)=J_1(x, dy)$ on $\sX \setminus \{x\}$ for $m$-a.e. $x\in \sX$,
then $(J_1 (x, dy), t)$ is also a L\'evy system of $X$. 
Thus it follows from \eqref{e:2.9} and \eqref{e:2.11} that by modifying the values of $J(x, y)$ outside a set having zero $m\times m$  measure, 
we may and do assume that the L\'evy system $(J(x,y)  m(dy), t)$ satisfies \eqref{e:2.9} and \eqref{e:2.11} everywhere on $\sX\times \sX\setminus {\rm diagonal}$;
that is, ${\bf J}_\phi$ holds. 
\end{proof}

\begin{remark} \rm 
In the proof of $ {\rm VD}+{\rm\bf HK}(\phi, t_0)\Longrightarrow{\bf J}_{\phi}$, we do not need the assumption that $X$ is conservative that is imposed in
${\rm\bf HK}(\phi, t_0)$.   \qed 
\end{remark}

 Next, the following lemma asserts that 
\[
{\rm VD}+ {\rm RVD}+ {\bf J}_{\phi}\Longrightarrow {\bf (Jc)}+ {\rm \bf (Jt)}_{\phi, \bar r} 
\]
with $\bar{r}=\diam(\sX)/\lambda_0$, where
 $\lambda_0 \geq 2 $ is  the constant in {\rm RVD} \eqref{e:rvd}.

\begin{lemma}\label{lemma22}
\begin{enumerate}[\rm (a)]
\item ${\rm VD}+{\bf J}_{\phi}\Longrightarrow {\bf (Jc)}$. 

\item   Suppose that $\phi$ is reverse doubling. Then  ${\rm VD}+{\bf J}_{\phi,\leq}\Longrightarrow {\bf(Jt)}_{\phi,\leq}$. 

\item ${\rm VD}+{\rm RVD}+{\bf J}_{\phi,\geq}\Longrightarrow
{\rm \bf (Jt)}_{\phi, \bar r, \geq}$ with $\bar r= \diam(\sX)/\lambda_0$,
where  $\lambda_0 \geq 2$ is  the constant in {\rm RVD} \eqref{e:rvd}. 
\end{enumerate}
\end{lemma}

\begin{proof}
(a). For $z\in\sX$ and $x,y\in\sX\setminus\{z\}$, 
\begin{align*} 
J(x, z)&\leq \frac{C_1}{V(x, d(x, z)) \phi(d(x, z))}\\
&=\frac{C_1}{V(y, d(y, z)) \phi (d(y, z))} \, \frac{V(y, d(y, z)) \phi (d(y, z))}{V(x, d(x, z)) \phi (d(x, z))}\\
&\leq \frac{C_1}{V(y, d(y, z)) \phi (d(y, z))} \, \frac{V(x,d(x,z)+2d(y, z)) \phi (d(y, z))}{V(x, d(x, z)) \phi (d(x, z))}  \\
&\leq C_2\,J(y,z) \left( 1+2\frac{d(y,z)}{d(x,z)} \right)^{d_1}\, \left( 1\vee\frac{d(y,z)}{d(x,z)} \right)^{ \beta },
 \end{align*} 
where the first inequality holds by $\bf J_{\phi,\leq}$, the second inequality holds as 
$$
B(x,d(x,z)+2d(y,z))\supset B(x,d(x,y)+d(y,z))\supset B(y,d(y,z)), 
$$
and the last inequality holds by $\bf J_{\phi,\geq}$, \eqref{e:vd} and \eqref{eqnphi}.  

\smallskip

(b) follows from ${\bf J}_{\phi,\leq}$ and Lemma \ref{L:2.1}.

\smallskip

(c)  By VD and RVD, we have   
\[
m\big(B(x,\lambda_0 r)\setminus B(x,r)\big)\geq C_3V(x,\lambda_0r)\quad\hbox{ for }x\in\sX,\,r<\diam(\sX)/\lambda_0. 
\]
So, for $x\in\sX$ and $r<\diam(\sX)/\lambda_0$, by the above the inequality  and \eqref{eqnphi},
\begin{align*} 
J(x,B(x,r)^c)&\geq\int_{B(x,\lambda_0r)\setminus B(x,r)}J(x,y)m(dy)\\
&\geq \int_{B(x,\lambda_0r)\setminus B(x,r)} \frac{C_4}{V(x, d(x, y)) \phi (d(x, y))} m(dy) \\
 &\geq  m\big(B(x,\lambda_0r)\setminus B(x,r)\big)\frac{  C_4}{V(x,\lambda_0r)\phi(\lambda_0r)}\\
 &\geq \frac{  C_3C_4}{\phi(\lambda_0r)}\geq \frac{  C_5}{\phi(r)}.
\end{align*}
\end{proof}

For the proof of Theorem \ref{T:1.3}, we also need a two-sided  mean exit time estimates ${\bf E}_{\phi,r_0}$. 

\begin{definition} \rm 
\begin{enumerate} [(i)]
\item  We say that condition ${\bf E}_{\phi,r_0}$ holds if there are constants $C_1,C_2>0$ and $  r_0\in(0,\infty]$ such that for every
$x\in \sX$ and $0<r <r_0$, 
\begin{equation*}  
C_1\phi(r)\leq\bE_x[\tau_{B(x,r)}]\leq C_2\phi(r).
\end{equation*}

\item  We say that condition ${\bf E}_{\phi,r_0,\leq}$ (resp. ${\bf E}_{\phi,r_0,\geq}$) holds if the upper bound (resp. lower
bound) in the above formula holds.
\end{enumerate}
\end{definition}

\begin{lemma}\label{lemma:E}
\begin{enumerate}[\rm (a)]
\item ${\bf (EP)} _{\phi, r_0,   \leq} \Longrightarrow  {\bf E}_{\phi,r_0,\geq}$. 

\item ${\rm {\bf (Jt)}_{\phi, \bar r, \geq}}\Longrightarrow  {\bf E}_{\phi,\bar r/2,\leq}$. 
\end{enumerate}
In particular, ${\bf (EP)} _{\phi, r_0,   \leq}  + {\bf (Jt)}_{\phi, \bar r, \geq} \Longrightarrow  {\bf E}_{\phi,r_0}$ with $r= r_0 \wedge (\bar r/2)$. 
\end{lemma}

\begin{proof}
(a) Suppose ${\bf (EP)} _{\phi, r_0,   \leq}$ holds and $C$ is the positive constant in \eqref{e:ep}. For $r\in (0,  r_0)$, take $t_1=\phi(r)/(2C)$. Then  by ${\bf (EP)}_{\phi, r_0,   \leq}$, 
$$
\bP_x\big(\tau_{B(x,r)}< t_1 \big)\leq\frac{C\,t_1}{ \phi(r)} =\frac12 \quad \hbox{for }x\in\sX .
$$
Hence
$$ 
\bE_x[\tau_{B(x,r)}]\geq t_1\bP_x(\tau_{B(x,r)}\geq t_1)\geq\frac{t_1}2=\frac{\phi(r)}{4C}.
$$

(b)  By  the L\'evy sytem of $X$ and ${\bf (Jt)}_{\phi, \bar r, \geq}$,
 we have for $x\in \sX$ and $0<r<\bar{r}/2$, 
\begin{align*}
1 &\geq  \bP_x( X_{\tau_{B(x, r)}} \in B(x,r)^c;X_{\tau_{B(x,r)}-}\neq X_{\tau_{B(x, r)}})\\
&= \bE_x \int_0^{\tau_{B(x, r)}} J(X_s, B(x,r)^c) ds  \geq  \bE_x \int_0^{\tau_{B(x, r)}} J(X_s, B(X_s,2r)^c) ds\geq  c \frac{\bE_x [\tau_{B(x, r)}]}{\phi(r)},
\end{align*}
where the last inequality is due to ${\rm {\bf (Jt)}_{\phi, \bar r, \geq}}$ and \eqref{eqnphi}. 
\end{proof}

We end this section by proving some simple exit distribution estimates. 

\begin{lemma}\label{lemma24}
\begin{enumerate} [\rm (a)]
\item  Suppose that $ {\bf (Jt)}_{\phi, \bar r, \leq}$  holds for some $\bar{r}\in(0,\infty]$. Then, there is a constant $C\in  (0,\infty)$ that depends only on the parameters in ${\bf (Jt)}_{\phi, \bar r, \leq}$ 
such that for every open $U\subset \sX$ and every $W\subset\sX\setminus \overline{U}$, we have 
\[
\bP_x(X_{\tau_U}\in W)\leq C \frac{\bE_x[\tau_U]}{\phi(d(U,W) \wedge\bar{r})} \quad \hbox{ for every }x\in U.
\]

\item   Suppose that ${\bf (Jc)}_{\bar{r}}$ and ${\bf (Jt)}_{\phi, \bar r,\geq}$ hold for some $\bar{r}\in(0,\infty]$. Then,  there are constants $C_1\in(0,\infty)$ and $C_2\in(2,\infty)$ that depend only on the parameters  in ${\bf (Jc)}_{\bar{r}}$ and ${\bf (Jt)}_{\phi, \bar r,\geq}$ such that 
\[
\bP_x\big(X_{\tau_{B(x,r)}}\in \big(B(y, s)\cup B(x,2r)\big)^c\big)\geq C_1\frac{\bE_x[\tau_{B(x,r)}]}{\phi(r)}
\]
for every   $s \in (0, \bar{r})$,   $r\in (0, \bar{r}/2)$ and $x,y\in \sX$  with  $d(x,y)>s+C_2r$.  

\item Suppose that  \eqref{e:Jc1}  and ${\bf (Jt)}_{\phi, \bar r, \geq}$ hold for some $\bar{r}\in(0,\infty]$. Then, there is a constant $C\in(0,\infty)$ that depends only on the parameters in \eqref{e:Jc1}   and ${\bf (Jt)}_{\phi, \bar r, \geq}$ such that 
\[
\bP_x\big(X_{\tau_{B(x,s)}}\in B(\xi,2r)^c\big)\geq C \left(\frac{s}{r} \right)^\theta\frac{\bE_x[\tau_{B(x,s)}]}{\phi(r)} 
\]
for $\xi\in\sX$, $0<s<r<\bar{r}/2$ and $x\in B(\xi,2r-2s)$, where $\theta$ is the exponent in  \eqref{e:Jc1}.
 \end{enumerate}
\end{lemma}

\begin{proof}
(a)   By using the L\'evy system of $X$, and by ${\bf (Jt)}_{\phi, \bar r, \leq}$,  
\begin{align*}
\bP_x(X_{\tau_U}\in W)&=\bE_x \int_0^{\tau_U}J(X_t,W)dt\leq \bE_x \int_0^{\tau_U}J(X_t,B(X_t,d(U,W)\wedge\bar{r})^c)dt\\
&\leq \bE_x \int_0^{\tau_U} \frac{C}{\phi(d(U,W)\wedge\bar{r})}dt =C\frac{\bE_x[\tau_U]}{\phi(d(U,W)\wedge\bar{r})}. 
\end{align*}

\smallskip

 (b)  By \eqref{e:Jc2}, there are $c_1\in(0,1)$ and $c_2>2$ so that 
\[
J(x,B(y,s))\leq c_1\,J(x,B(x,2r)^c)
\]
for every   $s\in (0, \bar{r})$,  $r \in (0, \bar{r}/2)$  and  $x,y\in\sX$ with  $d(x,y)>s+c_2r$. 
 Then by \eqref{e:Jc1}, ${\bf (Jt)}_{\phi,\bar{r},\geq}$ and \eqref{eqnphi}, there are positive constants $c_3,c_4,c_5$ so that 
 for every     $s\in (0, \bar{r})$,  $r \in (0, \bar{r}/2)$, 
 $x,y \in \sX$ with 
  $d(x,y)>s+c_2r$ and $z\in B(x,r)$ 
\begin{align*}
&\quad\ J\big(z,\big(B(y,s)\cup B(x,2r)\big)^c\big)\\
&\geq c_3  J\big(x,\big(B(y,s)\cup B(x,2r)\big)^c\big)=c_3 \left( J(x,B(x,2r)^c)- J(x,B(y,s)) \right) \\
&\geq c_3(1-c_1)J(x,B(x,2r)^c)\geq c_3c_4(1-c_1)/\phi(2r)\geq c_5/\phi(r) . 
\end{align*}
 
 Let $C_1:=c_5$ and $C_2:=c_2$. We have  by  the  L\'evy system formula and the above inequality that for any $x,y\in \sX$, $s<\bar{r}$ and  $r<\bar{r}/2$ such that $d(x,y)>s+C_2r$,
\begin{align*}
\bP_x\big(X_{\tau_{B(x,r)}}\in \big(B(y,s)\cup B(x,2r)\big)^c\big)
&=\bE_x\int_0^{\tau_{B(x,r)}} J\big(X_t, \big(B(y,s)\cup B(x,2r)\big)^c\big)dt\\
&\geq  \frac{C_1 \bE_x [ \tau_{B(x,r)}] } {\phi (r)}. 
\end{align*}  

\smallskip

(c)  For every  $0<s<r<\bar{r}/2$,  $\xi\in\sX$, $x\in B(\xi,2r-2s)$ and $y\in B(x,s)$, 
we have by  \eqref{e:Jc1}, ${\bf (Jt)}_{\phi, \bar r, \geq}$   and \eqref{eqnphi},
 \begin{eqnarray*}
&&  \int_{B(\xi,2r)^c}J(y,dz)
\geq \int_{B(\xi,2r)^c} c_6  \left(\frac{d(y,z)}{d(\xi,z)+d(y,z)} \right)^\theta J(\xi,dz)\\
&\geq & c_6 \int_{B(\xi, 2r)^c} \left(
\frac{ (d(\xi,z)-2r) +s}{d(\xi,z)+d(\xi,z)+2r}\right)^\theta J(\xi,dz)\\ 
&=& c_6 \int_{B(\xi, 2r)^c} \left(\frac{ (d(\xi,z)-2r) +s}{ 2(d(\xi,z)-2r) +6r} \right)^\theta J(\xi,dz)\\
&\geq & c_6   \int_{B(\xi, 2r)^c}  2^{-\theta} \left(\frac{s}{3r } \right)^\theta J(\xi,dz)\\
&=& c_6\, 6^{-\theta}(s/r)^\theta J(\xi, B(\xi, 2r)^c)  \\
&\geq & c_6\, 6^{-\theta}(s/r)^\theta \frac{c_2}{\phi (2r)} \geq    (s/r)^\theta  \frac{c_7 }{\phi (r)},
\end{eqnarray*}
where in the second inequality, we used the fact that for $0<a<b$ and $c\geq 0$, $\frac{a+c}{b+c} \geq \frac{a}{b}$. 
Thus by  the L\'evy system again, 
 \begin{eqnarray*}
 \bP_x\big(X_{\tau_{B(x,s)}}\in B(\xi,2r)^c\big)
 = \bE_x \int_0^{\tau_{B(x,s)}} J(X_t, B(\xi, 2r)^c ) dt \geq c_7 (s/r)^\theta  \,  \frac{\bE_x [ \tau_{B(x,s)}]}{\phi (r)}.
\end{eqnarray*}
This establishes (c). 
 \end{proof}

\section{\bf Uniform boundary Harnack principle}\label{S:3}

We present the proof of Theorem \ref{T:1.3} in this section. The proof is divided into two major steps. 

\subsection{\bf A box method argument} \label{S:3.1}

In this subsection, we establish the following key exit distribution estimate.

\begin{proposition}\label{prop31}
Suppose that   {\rm ${\bf (EP)} _{\phi, r_0,   \leq}$} and 
${\bf (Jt)}_{\phi, \bar r, \leq}$ hold   for some $  r_0,\bar{r}\in (0, \infty]$.  Then, there is a constant $C\geq 1$ that depends only on  
the parameters in    {\rm ${\bf (EP)} _{\phi, r_0,   \leq}$} and $ {\bf (Jt)}_{\phi, \bar r, \leq}$  and \eqref{eqnphi} such that, for each open $D\subset\sX$, $\xi\in \sX$ and $r\in (0,r_0\wedge\bar{r})$,  
\[ 
\bP_x (X_{\tau_{B_D (\xi, r)}} \in D)=\bP_x\big(\tau_D>\tau_{B_D(\xi,r)}\big)\leq   C\frac{\mathbb{E}_x[\tau_{B_D(\xi,r)}]}{\phi(r)}\quad\hbox{ for every }x\in B_D(\xi,r/2).
\]
\end{proposition}
 
\medskip

For its proof, we first prepare some estimates. 

\medskip

\begin{lemma}\label{lemma32}
Assume that {\rm ${\bf (EP)} _{\phi, r_0,   \leq}$} holds for some $r_0\in (0, \infty]$. 
 \begin{enumerate}[\rm (a).]
\item There is a constant  $C\in [1,\infty)$ depending only on the parameters in {\rm ${\bf (EP)} _{\phi, r_0,   \leq}$} such that for each open $D\subset\sX$, $\xi\in\sX$ and $r\in (0, r_0 )$, it holds that for each $x\in B_D(\xi,3r/4)$ and  $r_1\in (0, r/4]$.
\begin{equation}\label{e:3.1a}
\bP_x\big(\tau_D>\tau_{B(x,r_1)}\big)\leq C\sqrt{ {\bE_x[\tau_{B_D(\xi,r)}]}/{\phi(r_1)}}.
\end{equation}

\item There is  a constant  $C\in [1,\infty)$ that depends only on the parameters in {\rm ${\bf (EP)} _{\phi, r_0,   \leq}$}  and \eqref{eqnphi} such that for each open $D\subset\sX$, $\xi\in\sX$ and $r\in (0,r_0)$, it holds that
\[
\bP_x\big(\tau_D>\tau_{B_D(\xi,r)}\big)\leq C\sqrt{{\mathbb{E}_x[\tau_{B_D(\xi,r)}]} / {\phi(r)}}
\quad \hbox{ for each }x\in B_D(\xi, 3r/4).
\]

\end{enumerate}
\end{lemma}

\begin{proof}
(a)  We fix $x\in B_D(\xi,3r/4)$  and $0<r_1 \leq r/4$. 
Clearly, $B_D(x,r_1)\subset B_D(\xi,r)$, so for every $t>0$,
\[
\bP_x(\tau_{B_D(x,r_1)}\geq t)\leq\frac{\mathbb{E}_x[\tau_{B_D(x,r_1)}]}{t}\leq \frac{\mathbb{E}_x[\tau_{B_D(\xi,r)}]}{t}.
\]
On the other hand, by ${\bf (EP)} _{\phi, r_0,   \leq}$, 
\[
\bP_x(\tau_{B(x,r_1)}\geq t)\geq 1-C_1\frac{t}{\phi(r_1)} \quad  \hbox{for } t>0.
\]
Noticing that $\tau_{B_D(x,r_1)}=\tau_{B(x,r_1)}\wedge \tau_D$, we have for  $t>0$, 
\begin{eqnarray*}
\bP_x\big(\tau_D<t\leq \tau_{B(x,r_1)}\big)
&=&\bP_x\big(\tau_{B_D(x,r_1)}<t\leq \tau_{B(x,r_1)}\big)\\
&=& \bP_x(\tau_{B(x,r_1)}\geq t)-\bP_x(\tau_{B_D(x,r_1)}  \geq t )  \\
&\geq & 1-C_1\frac{t}{\phi(r_1)}-\frac{\mathbb{E}_x[\tau_{B_D(\xi,r)}]}{t}.
\end{eqnarray*}
Hence for $t>0$, 
\begin{equation}\label{e:3.2} 
\bP_x\big(\tau_D>\tau_{B(x,r_1)}\big)\leq 1-\bP_x\big(\tau_D<t\leq \tau_{B(x,r_1)}\big)\leq C_1\frac{t}{\phi(r_1)}+\frac{\mathbb{E}_x[\tau_{B_D(\xi,r)}]}{t}.
\end{equation}
Taking $t=\sqrt{\phi(r_1)\bE_x[\tau_{B_D(\xi,r)}]}$ in \eqref{e:3.2} yields the desired estimate \eqref{e:3.1a}. 

\smallskip

(b) follows from (a) by taking $r_1=r/4$ and using \eqref{eqnphi}, as $B_D(x, r/4)\subset B_D(\xi, r)$ for $x\in B_D(\xi, 3r/4)$.
\end{proof}

 \begin{corollary}\label{coro33}
Assume that {\rm ${\bf (EP)} _{\phi, r_0,   \leq}$} holds for some $r_0\in (0, \infty]$. Then, there is $C>0$  that depends only on 
 the parameters in {\rm ${\bf (EP)} _{\phi, r_0,   \leq}$}  and \eqref{eqnphi}
such that for each open $D\subset\sX$, $\xi\in\sX$ and $r\in(0, r_0)$, it holds that  for each $x\in B_D(\xi,3r/4)$, 
\[
\frac{\mathbb{E}_x[\tau_{B_D(\xi,r)}]}{\phi(r)\bP_x\big(\tau_D>\tau_{B_D(\xi,r)}\big)}
\geq C\left(\bP_x\big(\tau_D>\tau_{B_D(\xi,r)}\big)+\frac{\mathbb{E}_x[\tau_{B_D(\xi,r)}]}{\phi(r)}\wedge1\right) .
\]
\end{corollary}

\begin{proof}
Let $C_1$ be the constant of Lemma \ref{lemma32}(b).  Applying  Lemma \ref{lemma32}(b)  twice and using the fact $\sqrt{a}\geq a\wedge 1$ for any $a>0$, we get for $x\in B_D(\xi, 3r/4)$, 
\begin{align*}
\frac{\mathbb{E}_x[\tau_{B_D(\xi,r)}]}{\phi(r)\bP_x\big(\tau_D>\tau_{B_D(\xi,r)}\big)}&\geq \frac1{C_1}\sqrt{{\mathbb{E}_x[\tau_{B_D(\xi,r)}]} / {\phi(r)}}\\
&=\frac{1}{2C_1}\sqrt{{\mathbb{E}_x[\tau_{B_D(\xi,r)}]} / {\phi(r)}}+\frac{1}{2C_1} \sqrt{{\mathbb{E}_x[\tau_{B_D(\xi,r)}]} / {\phi(r)}}\\
&\geq \frac{1}{2C_1^2}\bP_x\big(\tau_D>\tau_{B_D(\xi,r)}\big)+\frac{1}{2C_1}
\Big(\frac{\mathbb{E}_x[\tau_{B_D(\xi,r)}]}{\phi(r)}\wedge 1 \Big).
\end{align*}
\end{proof}

\medskip

Corollary \ref{coro33} implies that the inequality of Proposition \ref{prop31} holds when $\bP_x\big(\tau_D>\tau_{B_D(\xi,r)}\big)+\frac{\mathbb{E}_x[\tau_{B_D(\xi,r)}]}{\phi(r)} \wedge 1 $ is large.
With these preparation,  we  next use the idea of box method which is developed by Bass and Burdzy \cite{BB1} and then adapted analytically in Aikawa \cite{Ai}, to finish the proof of Proposition \ref{prop31}.  Recall that $\sum_{k=1}^\infty k^{-2}= \pi^2/6$. 
 
 \medskip    
 
\begin{proof}[Proof of Proposition \ref{prop31}]
First, we introduce some subsets of $B_D(\xi,3r/4)$ as follows. 
  $$
U_0  :=\Big\{x\in B_D(\xi,3r/4): \bP_x(\tau_D>\tau_{B_D(\xi,r)})+\frac{\mathbb{E}_x[\tau_{B_D(\xi,r)}]}{\phi(r)}\geq \frac12\Big\},
$$
and for $j\geq 1$, 
\begin{align*}
U_j& :=\Big\{x\in B_D \Big( \xi,\frac{3r}4-\sum_{k=1}^j\frac{3r}{2\pi^2k^2} \Big): \bP_x(\tau_D>\tau_{B_D(\xi,r)})
+\frac{\mathbb{E}_x[\tau_{B_D(\xi,r)}]}{\phi(r)}
\in [2^{-(j+1)},2^{-j}) \Big\} ,\\
V_j&:=\bigcup_{k=0}^{j-1}U_k \qquad \hbox{and} \qquad 
W_j :=B_D \Big( \xi,\frac{3r}4 -\sum_{k=1}^{j-1}\frac{3r}{2\pi^2k^2} \Big) \setminus V_j.    
\end{align*}
Observe that  for $j\geq 0$, 
\begin{align}\label{e:3.1}
\begin{split}
V_j &  \supset   \Big\{x\in B_D \Big( \xi,\frac{3r}4-\sum_{k=1}^{j-1}\frac{3r}{2\pi^2k^2} \Big): \bP_x(\tau_D>\tau_{B_D(\xi,r)})
+\frac{\mathbb{E}_x[\tau_{B_D(\xi,r)}]}{\phi(r)}   
  \geq 2^{-j} \Big\}, \\
W_j &=  \Big\{x\in B_D \Big( \xi,\frac{3r}4-\sum_{k=1}^{j-1}\frac{3r}{2\pi^2k^2} \Big): \bP_x(\tau_D>\tau_{B_D(\xi,r)})
+\frac{\mathbb{E}_x[\tau_{B_D(\xi,r)}]}{\phi(r)}       < 2^{-j} \Big\},  
\end{split}
\end{align}
and
$$
B_D(\xi, r/2)\subset \bigcup_{j=1}^\infty V_j=\bigcup_{j=0}^\infty U_j \subset B_D(\xi, 3r/4).
$$ 
For $j\geq 1$, we define 
\begin{equation}\label{eqn32}
\lambda_j=\begin{cases}
	\inf\limits_{x\in V_j}\frac{\mathbb{E}_x[\tau_{B_D(\xi,r)}]}{\phi(r)\bP_x(\tau_D>\tau_{B_D(\xi,r)})}
	 &\hbox{ if }V_j\neq\emptyset,\\
	+\infty          &\hbox{ if }V_j=\emptyset.
\end{cases}
\end{equation}
For each $j\geq 1$ and $x\in U_j$, define $\tau_{x,j}:=\tau_{B_D(x,\frac{3r}{4\pi^2j^2})}$.

 \medskip

\noindent(i) By \eqref{eqn32} and the strong Markov property of $X$, we have  for $x\in U_j$, 
\begin{eqnarray}\label{eqn33}
 \bP_x(X_{{ \tau_{x,j} } }\in V_j,\tau_D>\tau_{B_D(\xi,r)})
 &=& \bE_x[\bP_{X_{{ \tau_{x,j} } }}(\tau_D>\tau_{B_D(\xi,r)});X_{{ \tau_{x,j} } }\in V_j]
 \nonumber \\
&\leq &  \frac{1} {\lambda_j \phi (r)} \bE_x \left[   {\mathbb{E}_{X_{{ \tau_{x,j} } }}[\tau_{B_D(\xi,r)}]}  ;  X_{{ \tau_{x,j} } }\in V_j  \right]
 \nonumber \\
 &\leq &  \frac{\bE_x[\tau_{B_D(\xi,r)}-{ \tau_{x,j} } ]}{ \lambda_j  \phi(r)}.
 \end{eqnarray}

\medskip

\noindent(ii) Since $x\in U_j\subset B_D(\xi,\frac{3r}4-\frac{3r}{2\pi^2})$,  $B_D(x,\frac{3r}{4\pi^2j^2})\subset B_D(\xi,\frac{3r}{4} -\frac{3r}{4\pi^2})$ and so the distance between $B_D(x,\frac{3r}{4\pi^2j^2})$ and $B_D(\xi, {3r}/4)^c$ is at least $\frac{3r}{4\pi^2}$. Thus 
 by Lemma \ref{lemma24}(a) and \eqref{eqnphi}, 
\begin{align}\label{eqn34}
\bP_x\big(X_{{ \tau_{x,j} } }\in D\setminus B_D(\xi, {3r}/4)\big)\leq C_1\frac{\bE_x[{ \tau_{x,j} } ]}{\phi(\frac{3r}{4\pi^2})}\leq C_2\frac{\bE_x[{ \tau_{x,j} } ]}{\phi(r)} . 
\end{align}

\medskip

\noindent(iii) For $2\leq i\leq j$  and $x\in U_j\subset B_D(\xi,\frac{3r}4-\sum_{k=1}^j\frac{3r}{2\pi^2k^2})$, noticing that
\begin{eqnarray*}
d\Big(B_D \big(x,\frac{3 r}{4\pi^2j^2}  \big),W_{i-1}\setminus (V_i\cup W_i )\Big)
& \geq &d\Big(x,  W_{i-1}\setminus B\big(\xi,\frac{3r}4-\sum_{k=1}^{i-1} \frac{3r}{2\pi^2k^2} \big)
\Big)-\frac{3r}{4\pi^2j^2} \\
& \geq &\sum_{k=i}^{j}\frac{3r}{2\pi^2k^2}-\frac{3r}{4\pi^2j^2}\geq \frac{3r}{4\pi^2i^2},
\end{eqnarray*}
we have by Lemma \ref{lemma24}(a) and \eqref{eqnphi},  
\[
\bP_x\big(X_{{ \tau_{x,j} } }\in W_{i-1}\setminus (V_i\cup W_i)\big)\leq C_3\frac{\bE_x[{ \tau_{x,j} } ]}{\phi(\frac{3r}{4\pi^2i^2})}
\leq C_4i^{2 \beta }\frac{\bE_x[{ \tau_{x,j} } ]}{\phi(r)}.
\]
Hence, by    the strong Markov property of $X$ and \eqref{e:3.1},   
\begin{eqnarray}\label{eqn35}
&& \bP_x\big(X_{{ \tau_{x,j} } }\in W_{i-1}\setminus (V_i\cup W_i),\tau_D>\tau_{B_D(\xi,r)}\big) \nonumber \\
&=& \bE_x \Big[  \bP_{X_{\tau_{x, j}}} ( \tau_D>\tau_{B_D(\xi,r)}); \, X_{{ \tau_{x,j} } }\in W_{i-1}\setminus (V_i\cup W_i) \big] 
\nonumber \\
&\leq & C_42^{1-i}i^{2  \beta }\frac{\bE_x[{ \tau_{x,j} } ]}{\phi(r)}.
\end{eqnarray}

\noindent(iv) By Lemma \ref{lemma32}(a), \eqref{eqnphi} and the definition of $U_j$, there are constants $C_5,C_6 
\geq 1 $ depending only on the parameters in \eqref{eqnphi} and {\rm ${\bf (EP)} _{\phi, r_0,   \leq}$} so that for $x\in U_j$,
\begin{align}\label{eqn36}
\begin{split}
&\quad\,\bP_x\big(X_{{ \tau_{x,j} } }\in D\big)=\bP_x\big(\tau_D>\tau_{B_D(x,\frac{3}{4\pi^2j^2}r)}\big)
\leq C_5\sqrt{\frac{\bE_x[\tau_{B_D(\xi,r)}]}{\phi(\frac{3}{4\pi^2j^2}r)}    }
 \\
 &  = C_5\sqrt{\frac{\bE_x[\tau_{B_D(\xi,r)}]}{\phi(r)}\frac{\phi(r)}{\phi(\frac{3}{4\pi^2j^2}r)}} 
 \leq C_62^{-j/2}j^{\beta }. 
\end{split}
\end{align}
 Hence by the strong Markov property of $X$,   \eqref{e:3.1} for $W_j$ and the definition of $U_j$,  
\begin{eqnarray}\label{eqn37}
 && \bP_x\big(X_{{ \tau_{x,j} } }\in W_j,\tau_D>\tau_{B_D(\xi,r)}\big)  \nonumber \\
&=&\bE_x\big[\bP_{X_{{ \tau_{x,j} } }}(\tau_D > \tau_{B_D(\xi, r)}) ; X_{{ \tau_{x,j} } } \in W_j\big] \nonumber \\
&\leq& \bP_x\big(X_{{ \tau_{x,j} } }\in W_j\big)\cdot 2^{-j}   \nonumber \\
&\leq & \bP_x\big(X_{{ \tau_{x,j} } }\in D \big) \cdot  2 \Big(\bP_x(\tau_D>\tau_{B_D(\xi,r)})+\frac{\bE_x[\tau_{B_D(\xi,r)}]}{\phi(r)}
\Big) \nonumber \\ 
&\leq&  2C_62^{-j/2}j^{\beta } \Big( \bP_x(\tau_D>\tau_{B_D(\xi,r)})+\frac{\bE_x[\tau_{B_D(\xi,r)}]}{\phi(r)} \Big),
 \end{eqnarray}
where we used  \eqref{eqn36} in the last inequality.

\medskip

Recall that $V_1\subset V_2\subset\cdots\subset V_j$ and $V_1\cup W_1=B_D(\xi, {3r}/4)$, we have
\begin{align*}
V_j\cup\big(\bigcup_{i=2}^{j}W_{i-1}\setminus (W_i\cup V_i)\big)\cup W_j
=V_j\cup \big(\bigcup_{i=2}^{j}(W_{i-1}\setminus W_i)\big)\cup W_j=V_j\cup W_1=B_D(\xi, 3r/4 ).
\end{align*}
Thus it follows from \eqref{eqn33},  \eqref{eqn34}, \eqref{eqn35} and  \eqref{eqn37}  that 
\begin{align*}
&\quad\,\bP_x(\tau_D>\tau_{B_D(\xi,r)})\\
&=\bP_x(X_{{ \tau_{x,j} } }\in V_j,\tau_D>\tau_{B_D(\xi,r)})+\sum_{i=2}^j\bP_x\big(X_{{ \tau_{x,j} } }\in W_{i-1}\setminus (W_i\cup V_i),\tau_D>\tau_{B_D(\xi,r)}\big)\\
		&\quad\qquad +\bP_x(X_{{ \tau_{x,j} } }\in W_j,\tau_D>\tau_{B_D(\xi,r)})+\bP_x\big(X_{{ \tau_{x,j} } }\in D\setminus B_D(\xi,  3r/4),\tau_D>\tau_{B_D(\xi,r)}\big)\\
		&\leq \lambda_j^{-1}\frac{\bE_x[\tau_{B_D(\xi,r)}-{ \tau_{x,j} } ]}{\phi(r)}+\sum_{i=2}^j C_42^{1-i}i^{2 \beta }\frac{\bE_x[{ \tau_{x,j} } ]}{\phi(r)}+C_2\frac{\bE_x[{ \tau_{x,j} } ]}{\phi(r)}\\
		&\quad  +2C_62^{-j/2}j^{ \beta }\Big(\bP_x(\tau_D>\tau_{B_D(\xi,r)})+\frac{\bE_x[\tau_{B_D(\xi,r)}]}{\phi(r)}\Big)\\
		&\leq \max\{\lambda_j^{-1},C_7\}\frac{\bE_x[\tau_{B_D(\xi,r)}]}{\phi(r)}+2C_62^{-j/2}j^{ \beta }
		\Big(\bP_x(\tau_D>\tau_{B_D(\xi,r)})+\frac{\bE_x[\tau_{B_D(\xi,r)}]}{\phi(r)}\Big)\\
		&\leq(1+2C_62^{-j/2}j^{ \beta })\max\{\lambda_j^{-1},C_7,1\}\frac{\bE_x[\tau_{B_D(\xi,r)}]}{\phi(r)}+2C_62^{-j/2}j^{ \beta }\bP_x(\tau_D>\tau_{B_D(\xi,r)}), 
\end{align*}
where  $C_7:=C_2+\sum_{i=2}^{\infty}C_42^{1-i}i^{2 \beta }$. Since the above estimate works for any $x\in U_j$, we have by the definition of $\lambda_{j+1}$,
\begin{equation}\label{eqn38}
\lambda_{j+1}\geq \frac{1-2C_62^{-j/2}j^{ \beta }}{1+2C_62^{-j/2}j^{ \beta }}\min\{\lambda_j,C_8\},
\end{equation}
where $C_8=\min\{1,C_7^{-1}\}$. Now,  fix $j_0$ such that  $2C_62^{-j/2}j^{ \beta }<1/2$ for every $j\geq j_0$. 
By Corollary \ref{coro33} and the fact that $\bP_x(\tau_D>\tau_{B_D(\xi,r)})+\frac{\bE_x[\tau_{B_D(\xi,r)}]}{\phi(r)} \geq 2^{-j}$ for each $x\in V_j$, there is a constant $C_9>0$ that depends only on the parameters in {\rm ${\bf (EP)} _{\phi, r_0,   \leq}$}  and \eqref{eqnphi} so that
\[
\lambda_j\geq C_92^{-j}\ \hbox{ for every }j\leq j_0.
\]
  For $j>j_0$, by using \eqref{eqn38}, we get 
\[
\lambda_j\geq \min\{C_92^{-j_0},C_8\}\cdot\prod_{k=j_0}^{j-1}\frac{1-2C_62^{-j/2}j^{ \beta }}{1+2C_62^{-j/2}j^{ \beta }}. 
\]
This proves $\lambda_j\geq \min\{C_92^{-j_0},C_8\}\cdot\prod_{k=j_0}^\infty \frac{1-2C_62^{-j/2}j^{ \beta }}{1+2C_62^{-j/2}j^{ \beta }} > 0 $ for every $j\geq 1$. The proposition follows immediately. 
\end{proof}

\subsection{\bf Boundary Harnack principle} \label{S:3.2}
 In this subsection, we give the proof of Theorem \ref{T:1.3}

\begin{proof}[Proof of Theorem \ref{T:1.3}]
 For simplicity, we prove the theorem for $ \kappa=1 $. The proof for other $\kappa \in (0, 2)$ is similar  with some obvious adjustments. Fix $D,\xi,r$ as in the statement.  Let $\gamma:[0,\infty)\to [0, r/2]$ be the decreasing function defined as
\[
\gamma(s)=\begin{cases}
	\frac{1}{8}(2-\frac{s}r)^2r\quad&\hbox{ if }s<2r,\\
	0\quad&\hbox{ if }s\geq 2r.
\end{cases}
\]
Observe that $s+2\gamma(s)<2r$ for every $s\in [0,2r)$. Thus  
\begin{equation}\label{e:3.11} 
  d(\xi,y)+2\gamma(d(\xi,y))<2r
 \quad \hbox{for every } y\in B_D(\xi,2r).
\end{equation} 
Let 
\[
T :=\tau_{B_D(X_0,\gamma(d(X_0,\xi)))}=\tau_D\wedge\inf\{t\geq 0:\,d(X_t,X_0)\geq \gamma(d(X_0,\xi))\}.
\]
Define $T_0  :=0$, $Y_0 :=X_0$, and iteratively, 
\[
T_n :=T\circ \theta_{T_{n-1}}+T_{n-1} \quad \hbox{ and } \quad Y_n :=X_{T_n} \quad \hbox{ for }n\geq 1.
\]
Then  $\{T_k; k\geq 0\}$ is a sequence of a stopping times and 
 $\{Y_k; k\geq 0 \}$  is a Markov chain. For $n\geq 1$, we introduce an event 
\begin{align*}
 \mathcal{A}_n:=\big\{Y_n\in B_D(\xi,3r/2)\cup B_D(Y_{n-1},2R_{n-1}) \big\},
\end{align*}
where we use the convention $B(x,0):=\emptyset$ and denote 
\[
R_n:=\gamma\big(d(Y_n,\xi)\big)\quad  \hbox{ for every }n\geq 0. 
\]
We use $\mathcal{A}_n^c$ to denote the complement of $\mathcal{A}_n$. We divide the proof into four steps. \medskip 

(i)  We claim that for $n\geq 0$, $Y_n\in B_D(\xi,2r-\frac{r}{n+1}) $ and so, by the definition of  $\gamma$,  $R_n\geq\frac{r}{8(n+1)^2}$   on  $\{X_0\in B_D(\xi,r)\}\cap (\bigcap_{k=1}^n\mathcal{A}_k)$.
 
\smallskip

It suffices to show that $2r-d(Y_n,\xi)>\frac{r}{n+1}$ on $\{X_0\in B_D(\xi,r)\}\cap (\bigcap_{k=1}^n\mathcal{A}_k)$ for each $n\geq 0$. We do this by mathematical induction. Clearly on $\{X_0\in B_D(\xi,r)\}$,  $2r-d(Y_0,\xi)>r$ so the claim holds for $n=0$. Next, suppose  that the claim holds for $n=j-1$ for some $j\geq 1$.
 Then on $\{X_0\in B_D(\xi,r)\}\cap (\bigcap_{k=1}^j\mathcal{A}_k)$,    either $d(Y_j,\xi)<3r/2$, or $d(Y_{j-1},Y_j)<2\gamma(d(Y_{j-1},\xi))$ 
 and so by induction hypothesis,  
\[
2r-d(Y_j,\xi)  \geq  2r-d(Y_{j-1},\xi)-  2\gamma(d(Y_{j-1},\xi)) >  \frac{r}{j}  -\frac{r}{4j^2}>\frac{r}{j+1}.	
\]
That is, the claim holds for $n=j$.  This proves the Claim (i).

\medskip

 (ii) We next show that  there is a constant $0<C_1<1$ depending only on the parameters in   \eqref{eqnphi}, ${\bf (Jc)}_{ \bar r}$, ${\bf (Jt)}_{\phi,\bar r,\geq}$ and ${\bf E}_{\phi,r_0,\geq}$ 
so that
\begin{equation}\label{eqn39}
\bP_x(\mathcal{A}_n|Y_{n-1})\leq 1-C_1\quad\hbox{ for every }x\in B_D(\xi,  2r ) \hbox{ and } n\geq 1.
\end{equation} 
 
 Noticing that $\{Y_{n-1}\notin B_D(\xi,2r)\}\cap\mathcal{A}_n=\emptyset$, we only need to prove the inequality on $\{Y_{n-1}\in B_D(\xi,2r)\}$.
  By the strong Markov property of $X$ at time $T_{n-1}$, it suffices to show 
\[
\bP_x(\mathcal{A}_1)\leq 1-C_1\quad \hbox{for every }x\in B_D(\xi,2r).
\]
Let $C_2 \in (0,\infty)$  and $C_3\in(2,\infty)$ be the constants of Lemma \ref{lemma24}(b). Take $C_4\in (0,1)$  so that $C_4 + C_3 C_4^2/8=1/2$.   
We have with $R_0:=\gamma(d(x,\xi))$,  
\begin{equation}\label{eqn310}
 d(x,\xi)\geq C_3R_0+3r/2 \quad  \hbox{ for every }x\in B(\xi,2r)\setminus B(\xi,2r-C_4r).
\end{equation}
Thus  for every $x\in  B_D(\xi,2r)\setminus B_D (\xi,2r-C_4r)$,   
\begin{align*}
\bP_x(\mathcal{A}_1)
& =\bP_x\big(X_{\tau_{B_D(x,R_0)}}\in B_D(\xi,3r/2)\cup B_D(x,2R_0)\big)\\
& =\bP_x\big(\tau_{B_D(x,R_0)}=\tau_{B(x,R_0)},X_{\tau_{B_D(x,R_0)}}\in B_D(\xi,3r/2)\cup B_D(x,2R_0)\big) \\
& \leq \bP_x\big(X_{\tau_{B(x,R_0)}}\in B_D(\xi,3r/2)\cup B_D(x,2R_0)\big)\\
&\leq  1-\bP_x\big(X_{\tau_{B(x,R_0)}}\in  \big(B_D(\xi,3r/2)\cup B_D(x,2R_0)\big)^c \big) \\
&\leq  1-C_2\frac{\bE_x[\tau_{B(x,R_0)}]}{\phi(R_0)} \\
&\leq 1-C_5,
\end{align*}
where we used \eqref{eqn310} in the second inequality,   Lemma \ref{lemma24}(b) in the third  inequality,  \eqref{eqnphi} and Lemma \ref{lemma:E}(a) in the last inequality. When  $x\in B_D(\xi,2r-C_4r)$,    
\begin{eqnarray*}
\bP_x(\mathcal{A}_1)
&\leq & \bP_x\big(X_{\tau_{B(x,R_0)}} \in B_D(\xi,3r/2)\cup B_D(x, 2R_0)\big) \\
&\leq & 1-\bP_x\big(X_{\tau_{B(x,R_0)}}\in B(\xi,2r)^c\big) \\
& \leq & 1-C_6(R_0/ r)^\theta \, \frac{\bE_x[\tau_{B(x,R_0)}]}{\phi(r)}\\
&\leq & 1-C_7,
\end{eqnarray*}
where  we used \eqref{e:3.11} for the second inequality, Lemma \ref{lemma24}(c), \eqref{e:3.11} and the fact $r<( \bar r\wedge\bar{r})/2$ for the  third inequality, and \eqref{eqnphi}, the fact $R_0=\gamma(d(x,\xi))\geq C_4^2r/8$ and Lemma \ref{lemma:E}(a) for the last inequality. 
 This establishes \eqref{eqn39}.  As a consequence of \eqref{eqn39} and the strong Markov property of $X$, we have
\begin{equation}\label{eqn312}
\bP_x \left( \cap_{k=1}^n\mathcal{A}_k \right) \leq (1-C_1)^{n-1}\bP_x(\mathcal{A}_1)
\quad \hbox{for every } x\in B_D(\xi,2r)  \hbox{ and } n\geq 2.
\end{equation}

\medskip

(iii)  Let $h$ be a non-negative function that is regular harmonic with respect to $X$ in $B_D(\xi,2r)$ and vanishes on $B(\xi,2r)\setminus D$.
We claim that there is a positive constant $ C_8$ that depends only on the parameters in 
$\bf (Jc)_{ \bar r}$,  ${\bf (Jt)}_{\phi, \bar r, \geq}$, ${\bf E}_{\phi,r_0,\leq}$ and \eqref{eqnphi} so that for any $x\in B_D(\xi, r)$ and $n\geq 1$, on the event  $\mathcal{A}_1\cap \mathcal{A}_2\cap\cdots\cap \mathcal{A}_{n-1}$, 
\begin{equation}\label{eqn313}
\bE_x[h(Y_n)\cdot \1_{\mathcal{A}^c_{n}}\big|{ \sG_{T_{n-1}}}]\leq C_8  n^{ 2\theta}  \phi(r)
\int_{\sX\setminus B(\xi,3r/2)}h(y)J(\xi,dy),
\end{equation}
where $\theta$ is the exponent in $\bf (Jc)_{ \bar r}$. 

\smallskip

By the strong Markov property of $X$, 
$\bE_x[h(Y_n)\cdot \1_{\mathcal{A}^c_{n}}\big| \sG_{T_{n-1}}]=\bE_{Y_{n-1}}[h(Y_1)\cdot \1_{\mathcal{A}^c_1}]$.
By Step (i),  $Y_{n-1}\in B_D(\xi,2r-\frac{r}{n})$. So, it suffices to prove 
\begin{equation}\label{e:3.16a}
 \bE_y[h(Y_1)\cdot \1_{\mathcal{A}^c_1}]\leq C_8n^{2\theta}\phi(r)\int_{\sX\setminus B(\xi,3r/2)}h(y)J(\xi,dy)
\quad  \hbox{for  } y\in B_D(\xi,2r-\tfrac{r}{n}).
\end{equation}
For $y\in B_D(\xi,2r-\tfrac{r}{n})$, $w\in B_D(y, R_0)$ and $z\in B(y,2R_0)^c$, where $R_0:=\gamma(d(y,\xi))$,  we have 
\[
d(w, z)\geq R_0\geq\gamma(2r-\frac{r}{n})>\frac{r}{8n^2} . 
\]
Hence
\begin{align*}
d(w,z) > \frac12d(w,z)+\frac{r}{16n^2} 
 \geq \frac{1}{2}d(w,z)+\frac{1}{32n^2}d(w,\xi)\geq \frac{1}{32n^2}d(\xi,z),
\end{align*}
where the second inequality follows from $d(\xi,w)<2r$. By ${\bf (Jc)}_{ \bar r}$ and the above inequality, 
\begin{equation} \label{e:3.16}
J(w,  dz )\leq C_9\,n^{  2\theta}J(\xi,  dz )\  \hbox{ on }B(y,2R_0)^c \ \hbox{ for }y\in B_D(\xi,2r-\tfrac{r}{n}), w\in B_D(y, R_0).
 \end{equation} 
By the L\'evy system of $X$  and \eqref{e:3.16}, we have for   $y\in B_D(\xi,2r-\frac{r}{n})$, 
\begin{eqnarray} \label{e:3.18}
 \bE_y[h(Y_1)\cdot \1_{\mathcal{A}^c_1}]
&=& \bE_y \Big[\int_0^{T_1}\int_{B(\xi,3r/2)^c\cap B(y,2R_0)^c}h(z)J(X_t,dz) \Big]  \nonumber \\
&\leq &\bE_y[T_1]\cdot\sup_{w\in B_D(y,R_0)}\int_{B(\xi,3r/2)^c\cap B(y,2R_0)^c}h(z)J(w,dz)\nonumber  \\
&\leq & \bE_y[T_1]\cdot  C_9\,n^{2\theta}\int_{B(\xi,3r/2)^c\cap B(y,2R_0)^c}h(z) J(\xi,dz)   \\
&\leq & C_{10}\,\phi({r}/2)\,n^{2\theta}\int_{B(\xi,3r/2)^c\cap B(y,2R_0)^c}h(z) J(\xi,dz),\nonumber 
  \end{eqnarray}
where the last inequality holds by Lemma \ref{lemma:E}(b) and the fact $T_1:=\tau_{B_D(y,R_0)}\leq \tau_{B(y,r/2)}$.
 This establishes the inequality \eqref{e:3.16a}  in view of \eqref{eqnphi}. 

\medskip

 (iv) Now, suppose that  $h$ is  a bounded non-negative function that is regular harmonic with respect to $X$ in $B_D(\xi,2r)$ and vanishes on $B(\xi,2r)\setminus D$. By \eqref{eqn312},  $\bP_x \left( \cap_{k=1}^\infty\mathcal{A}_k \right) =0$ for every $x\in B_D(\xi, r)$. 
 Thus   for each $x\in B_D(\xi,  r )$,  by \eqref{e:3.18} for $n=1$, the strong Markov property of $X$, and \eqref{eqn312}-\eqref{eqn313},  with $R_0:= \gamma (d(x, \xi))$, 
\begin{align*}
h(x)
&=\bE_x[h(Y_1)] 
 =\bE_x[h(Y_2),\mathcal{A}_1]+\bE_x[h(Y_1),\mathcal{A}_1^c]\\
&=\sum_{n=1}^\infty \bE_x \left[ h(Y_n) \1_{\mathcal{A}_n^c};\cap_{k=1}^{n-1}\mathcal{A}_k \right]\\
&\leq C_9\,\bE_x[T_1]\int_{\sX\setminus B(\xi,3r/2)}h(y)J(\xi,dy)\\
&\qquad\quad+\sum_{n=2}^\infty(1-C_1)^{n-2}\bP_x(\mathcal{A}_1)\cdot C_8n^{ 2\theta}
\phi(r)\int_{\sX\setminus B(\xi,3r/2)}h(y)J(\xi,dy)\\
&\leq C_{11}  \left(\bE_x[T_1]+\bP_x(\mathcal{A}_1)    \phi (r)  \right)\int_{\sX\setminus B(\xi,3r/2)}h(y)J(\xi,dy).
\end{align*}
By Proposition \ref{prop31} (where we let $\xi=x$)  and \eqref{eqnphi}, we have  for every $x\in B_D(\xi,r)$,
\[
\bP_x(\mathcal{A}_1)\leq \bP_x(\tau_D>T_1)\leq C_{12}\frac{\bE_x[T_1]}{\phi(r)}.
\]
Hence we get for every $x\in B_D(\xi,r)$, 
\begin{equation}\label{eqn315}
h(x)\leq C_{11}(1+C_{12})\bE_x[T_1]\int_{\sX\setminus B(\xi,3r/2)}h(y)J(\xi,dy).
\end{equation}

For a general non-negative $h$  that is regular harmonic in $B_D(\xi,2r)$ and vanishes on $B(\xi,2r)\setminus D$, define for each $n\geq 1$, 
 $$
 h_n (x)=
 \begin{cases}  \bE_x [ (h\wedge n)(X_{\tau_{B_D(\xi,2r)}})] \quad &\hbox{for } x\in B_D(\xi, 2r), \\
 h(x)\wedge n &\hbox{for }x\in\sX\setminus B_D(\xi, 2r).
 \end{cases}
 $$
 Note that  $h_n$ is regular harmonic in $B_D(\xi,2r)$ and vanishes on $B(\xi,2r)\setminus D$, and $h_n$ increases to $h$ on $\sX$ as $n\to \infty$.  Thus we have by \eqref{eqn315} that for every $x\in B_D(\xi,r)$, 
\begin{align}\label{eqn316}
\begin{split}
h(x)&=\lim_{n\to \infty} h_n(x)\\
      &\leq\lim_{n\to \infty} C_{11}(1+C_{12})\bE_x[T_1]\int_{\sX\setminus B(\xi,3r/2)}h_n(y)J(\xi,dy)\\
     &\leq C_{11}(1+C_{12})\bE_x[T_1]\int_{\sX\setminus B(\xi,3r/2)}h(y)J(\xi,dy).
\end{split}
\end{align}

On the other hand, for $x\in B_D(\xi,r)$ and $z\in B_D(x,R_0)$ with $R_0:=\gamma(d(x,\xi))$, 
\[
d(\xi,z)<d(\xi,x)+R_0\leq  r+\gamma(r)=9r/8.
\] 
Hence for $x\in B_D(\xi,r)$, $y\in B(\xi,3r/2)^c$ and $z\in B_D(\xi,R_0)$, 
\[
 \frac{d(z,y)+d(\xi,y)}{d(\xi,y)}\leq \frac{d(z,\xi)+2d(\xi,y)}{d(\xi,y)}=\frac{d(z,\xi)}{d(\xi,y)}+2\leq\frac{9r/8}{3r/2}+2=\frac{11}{4}. 
\] 
Thus by the L\'evy system of $X$, ${\bf (Jc)}_{ \bar r}$ and \eqref{eqnphi}, we have for $x\in B_D(x, r)$, 
\begin{eqnarray}\label{eqn317}   
 h(x) &\geq &  \bE_x [h(X_{T_1}); X_{T_1}\in \sX \setminus B(\xi, 3r/2)]  \nonumber \\
&= & \bE_x \int_0^{T_1} \int_{\sX \setminus B(\xi, 3r/2)} h(y)J(X_s, dy) ds   \nonumber \\
&\geq & \bE_x[T_1] \inf_{z\in B(x,R_0)} \int_{\sX \setminus B(\xi, 3r/2)} h(y)J(z,dy) \nonumber \\
 &\geq &  \bE_x[T_1]  \int_{\sX \setminus B(\xi, 3r/2)}  h(y)\,C_{13}(\frac{d(\xi,y)}{d(\xi,y)+d(z,y)})^\theta J(\xi,dy)  \nonumber \\
&\geq &  C_{13}  ( {4}/ {11}  )^\theta. \, \bE_x[T_1]\int_{\sX\setminus B(\xi,3r/2)}h(y)J(\xi,dy), 
 \end{eqnarray}
where the constants $C_{13}$ and  $\theta$ are the corresponding parameters in  ${\bf (Jc)}_{ \bar r}$.

The conclusion of Theorem   \ref{T:1.3}]  now follows immediately from  \eqref{eqn316} and \eqref{eqn317}. 
\end{proof}

 \begin{remark}\label{R:3.4} \rm
 \begin{enumerate} [(i)]
\item In the proof of Theorem \ref{T:1.3}, we have the following approximate factorization result
for any non-negative function $h$ on $\sX$  that is regular harmonic in $B_D(\xi,2r)$ and vanishes on $B(\xi,2r)\setminus D$:
\[
h(x)\asymp \bE_x[\tau_{B_D(x,r/2})]\int_{ \sX\setminus B(\xi,3r/2)}h(y)J(\xi,dy)
\quad \hbox{for }x\in   B_D(\xi,2r/3)
\]
with the comparison constants depending only on the parameters in ${\bf (Jc)}_{ \bar r}$, ${\bf (Jt)}_{\phi, \bar r}$, ${\bf EP} _{\phi, r_0, \leq}$ and \eqref{eqnphi}. 
Indeed, the upper bound is immediate by \eqref{eqn316} as $ T_1\leq\tau_{B_D(x,r/2)}$, and the lower bound can be obtained by
   the same argument as that for  \eqref{eqn317}. More generally, by an easy modification of the parameters and $\gamma$ in the proof, 
   for positive constants $c_1,c_2,c_3$ such that $c_1+c_3<c_2<2$, we have for any non-negative function $h$  on $\sX$ 
   that is regular harmonic in $B_D(\xi,2r)$ and vanishes on $B(\xi,2r)\setminus D$, 
\begin{equation}\label{e:3.22}
h(x)\asymp \bE_x[\tau_{B_D(x,c_1r})]\int_{\sX\setminus B(\xi,c_2r)}h(y)J(\xi,dy)
\quad \hbox{for }x\in B_D(\xi,c_3r), 
\end{equation}
 with the comparison constants depending only on  $c_1,c_2,c_3$ and  the parameters in ${\bf (Jc)}_{ \bar r}$,  ${\bf (Jt)}_{\phi, \bar r}$,  ${\bf (EP)}_{\phi,r_0,\leq}$ and \eqref{eqnphi}.

\item For non-local operators and discontinuous Markov  processes with non-degenerate jump kernels, there are several works in literature that obtain  BHP through establishing  
 the approximate factorization \eqref{e:3.22} for non-negative  regular harmonic functions in $B_D(\xi,2r)$ and vanishes on $B(\xi,2r)\setminus D$.
See \cite{BKK1} for  isotropic stable processes, \cite{KK} for stable-subordinate Brownian motions on Sierpinski gasket, 
\cite{KS} for rotationally symmetric pure jump L\'evy processes on $\R^d$,  and  \cite{BKK} for a class of Feller processes in weak duality 
on metric measure spaces having strong Feller property   under a comparability condition  of
the jump kernel and a Urysohn-type  cufoff function condition on the domains of the generators of the processes and their duals.
\qed 
\end{enumerate}
 \end{remark}

\section{\bf Examples} \label{S:4}
Let $(\mcE,\mcF)$ be a quasi-regular Dirichlet form on $L^2(\sX; m)$, where $(\sX, d)$ is a Lusin space. It is well-known that $\mcE$ admits a  Beurling-Deny decomposition: there is a local symmetric form $\mcE^{(c)}$ on $\sF$, a symmetric positive $\sigma$-finite measure $J(dx,dy)$ on the product space $\sX\times\sX$ off the diagonal  and a $\sigma$-finite measure $\kappa$ on $\sX$ such that 
\[
\mcE(u,v)=\mcE^{(c)}(u,v)+\int_{\sX\times\sX\setminus {\rm \small diagonal}}\big(u(x)-u(y)\big)\big(v(x)-v(y)\big)J(dx,dy)+\int_\sX u(x)v(x)\kappa (dx)
\]
for each $u,v\in \mcF$. The measures $J$ and $\kappa$ are called the jump measure and the killing measure of $(\sE, \sF)$.
We   say that $(\mcE,\mcF)$ is of pure jump type if  $\mcE^{(c)}=0$. 

\begin{lemma}\label{lemmaj1}
Suppose that  $(\mcE,\mcF)$ is a quasi regular Dirichlet form on $L^2(\sX; m)$, where $(\sX, d)$ is a Lusin space and $m$ is a 
$\sigma$-finite measure with full support on $\sX$. Denote by $\kappa$ the killing measure of $(\sE, \sF)$.
 Suppose that  there is a symmetric $\sigma$-finite measure $\hat{J}(dx,dy)$ on $\sX\times\sX\setminus {\rm \small diagonal}$,\and constants $0<C_1<C_2<\infty$ such that 
\begin{align} \label{eqn41}
C_1\int_{\sX\times\sX\setminus {\rm \small diagonal}}&\big(u(x)-u(y)\big)^2\hat{J}(dx,dy)\\
&\leq \mcE(u,u)\leq C_2\int_{\sX\times\sX\setminus {\rm \small diagonal}}\big(u(x)-u(y)\big)^2\hat{J}(dx,dy)+C_2\int_\sX u^2(x)\kappa (dx)
\nonumber
\end{align}
for each $u\in \mcF_b$. Then $(\mcE,\mcF)$ is of pure jump type, i.e., its strongly local part $\mcE^{(c)}$ vanishes. 
\end{lemma}

\begin{proof} For $u\in \sF$, since $\mcE^{(c)}(u,u)^{1/2} \leq \mcE^{(c)}(u^+,u^+)^{1/2}+\mcE^{(c)}(v^-,v^-)^{1/2}$,
it suffices to show that $\sE^{(c)}(u, u)=0$ for every non-negative $u\in \sF_b$. 
 
For any non-negative $u\in \sF_b$, define
\[
u_{k,i}:= (u - (i-1)/k)^+\wedge (1/k) \quad  \hbox{ for }i\geq 1,\,k\geq 1.
\]
Note that for each $k\geq 1$, $0\leq u_{k,i}\leq 1/k$ for each $i\geq 1$ and $u=\sum_{i=1}^\infty u_{k,i}$. Clearly, on $\sX$, 
\begin{align}
\label{eqn42}&\lim\limits_{k\to\infty}\sum\limits_{i=1}^\infty u^2_{k,i}
\leq\lim\limits_{k\to\infty}\sum\limits_{i=1}^\infty\frac1  ku_{k,i}= \lim_{k\to \infty}  u/k  = 0,\\
\label{eqn43}&\sum\limits_{i=1}^\infty u^2_{k,i}\leq \sum_{i=1}^\infty u\cdot u_{k,i}=u^2. 
\end{align}
Moreover, for each $x,y\in \sX$ and $k\geq 1$,  $|u_{k,i}(x)-u_{k,i}(y)|\leq\frac1k$ for each $i\geq1$ and
$$
\big|u(x)-u(y)\big|=\big|\sum_{i=1}^\infty\big(u_{k,i}(x)-u_{k,i}(y)\big)\big|   =  
\sum_{i=1}^\infty\big|u_{k,i}(x)-u_{k,i}(y)\big|
$$ 
by considering the cases of $u(x)>u(y)$ and $u(x)\leq u(x)$ separately. 
 Hence 
\begin{align}
\label{eqn44}&\lim\limits_{k\to\infty}\sum\limits_{i=1}^\infty\big(u_{k,i}(x)-u_{k,i}(y)\big)^2\leq\lim\limits_{k\to\infty}\frac{1}{k} \sum\limits_{i=1}^\infty|u_{k,i}(x)-u_{k,i}(y)|=\lim\limits_{k\to\infty}\frac{1}{k}|u(x)-u(y)|=0,\\
\label{eqn45}&\sum\limits_{i=1}^\infty\big(u_{k,i}(x)-u_{k,i}(y)\big)^2\leq  \big|u(x)-u(y)\big|\sum\limits_{i=1}^\infty\big|u_{k,i}(x)-u_{k,i}(y)\big|=\big(u(x)-u(y)\big)^2.
\end{align}
Thus for each $k\geq 1$, as  $u_{k, i}=\frac1k$ on $\operatorname{supp}[u_{k, j}]$ for $1\leq i<j$, by the strong locality of $\sE^{(c)}$ 
(see \cite[Theorem 2.4.3]{CF}), 
\begin{align*}
\mcE^{(c)}(u,u)&=\sum_{i=1}^\infty \mcE^{(c)}(u_{k,i},u_{k,i})  
\leq  \sum_{i=1}^\infty \mcE(u_{k,i},u_{k,i})\\
&\leq C_2\sum_{i=1}^\infty\Big(\int_{\sX\times\sX\setminus {\rm \small diagonal}}\big(u_{k,i}(x)-u_{k,i}(y)\big)^2\hat{J}(dx,dy)+\int_\sX u_{k,i}^2(x)\kappa (dx)\Big)\\
&=C_2\int_{\sX\times\sX\setminus {\rm \small diagonal}}\sum_{i=1}^\infty\big(u_{k,i}(x)-u_{k,i}(y)\big)^2\hat{J}(dx,dy)
+C_2\int_{\sX}\sum_{i=1}^\infty u^2_{k,i}(x)\kappa (dx) . 
\end{align*}
 Then,  by \eqref{eqn41}-\eqref{eqn45} and the dominated convergence theorem, we conclude that $\mcE^{(c)}(u,u)=0$ by taking  $k\to\infty$ in the above inequality.
\end{proof}

\begin{corollary}\label{coroj2}
Suppose that $(\sX, d)$ is a locally compact separable metric space and $m$ is a Radon measure on $\sX$ with full supprt.
If $X$ is $m$-symmetric and  that {\bf HK$(\phi,t_0)$} holds for some $t_0\in(0,\infty]$, 
then the Dirichlet form $(\mcE,\mcF)$ associated with $X$ is regular on $L^2(\sX; m)$ and is of pure jump type admiting
 no killings inside $\sX$. 
Moreover, there is a symmetric function $J(x, y)$ on $\sX\times \sX\setminus {\rm diagonal}$ so that 
$J(dx, dy):=J(x, y) m(dx)m(dy)$ is the jump measure of $(\sE, \sF)$ and has the property $J_\phi$. Moreover, 
\begin{equation}\label{e:4.6}
\sF= \left\{ u\in L^2 (\sX; m): \int_{\sX\times \sX} (u(x)-u(y))^2 J(x, y) m(dx) m(dy) <\infty \right\}
\end{equation} 
and
$$
\sE (u, v) = \frac12 \int_{\sX\times \sX} (u(x)-u(y)) (v(x)-v(y)) J(x, y) m(dx) m(dy) 
\quad \hbox{for } u, v\in \sF.
$$
\end{corollary}
\begin{proof}
By the proof of \cite[Theorem 4.14]{CK} and  {\bf HK$(\phi,t_0)$}, we know that the transition density function $p(t,x,y)$ of $X$ has a jointly continuous modification on $(0, \infty) \times \sX \times \sX$. Consequently, $X$ is a Feller process on $(\sX,d)$ having strong Feller property. It follows that the Dirichlet form $(\mcE,\mcF)$ of $X$ is regular on $L^2(\sX; m)$. Denote by $\{P_t; t\geq 0\}$ the transition semigroup of $X$. 
It is known that 
\begin{eqnarray}
\sF &=& \left\{ u\in L^2(\sX;m): \sup_{t>  0}\frac1t(u-P_tu,u) < \infty \right\}, \nonumber \\
\sE(u, v)&=& \lim_{t\to 0} \frac1t(u-P_tu,v) \quad \hbox{for } u, v\in \sF.   \label{e:4.7} 
\end{eqnarray}   
It follows from {\bf HK$(\phi,t_0)$}  that there are constants $C_2>C_1>0$ so that for each $u\in L^2(\sX;m)$, 
\begin{align*}
C_1\int_{\sX\times\sX} \frac{\big(u(x)-u(y)\big)^2}{\phi\big(d(x,y)\big)V(x,d(x,y))}m(dx)&m(dy)\leq 
 \mcE(u,u) = \lim\limits_{t\to0}\frac1t(u-P_tu,u)  \\
&\leq C_2\int_{\sX\times\sX} \frac{\big(u(x)-u(y)\big)^2}{\phi\big(d(x,y)\big)V(x,d(x,y))}m(dx)m(dy).
\end{align*}
By Lemma \ref{lemmaj1}, 
 the Dirichlet form is of pure jump type. Since $X$ is conservative, the killing measure $\kappa =0$.
By \eqref{e:4.7} and {\bf HK$(\phi,t_0)$}, there are constants $C_4>C_3>0$ so that for any 
non-negative $u,v\in C_c(\sX)\cap \mcF$ with $\operatorname{supp}[u]\cap\operatorname{supp}[v]=\emptyset$,
\begin{align*}
C_3\int_{\sX\times\sX}&\frac{u(x)v(y)}{\phi\big(d(x,y)\big)V(x,d(x,y))}m(dx)m(dy)\\&\leq \int_{\sX\times\sX} u(x)v(y)J(dx,dy)\leq C_4\int_{\sX\times\sX} \frac{u(x)v(y)}{\phi\big(d(x,y)\big)V(x,d(x,y))}m(dx)m(dy)
\end{align*}
for non-negative $u,v\in C_c(\sX)\cap \mcF$ such that $\operatorname{supp}[u]\cap\operatorname{supp}[v]=\emptyset$. We conclude that  there is a symmetric kernel $J(x,y)$ such that $J(dx,dy)=J(x,y)m(dx)m(dy)$ and that
$$
\frac{C_3}{\phi\big(d(x,y)\big)V(x,d(x,y))} \leq J(x, y) 
\leq \frac{C_4}{\phi\big(d(x,y)\big)V(x,d(x,y))}
$$
for $m$-a.e. $x\not= y\in \sX$, that is, ${\bf J}_\phi$ holds. Note that in this case,  $(J(x,y)m(dy),t)$ is a L\'evy system of the Feller process $X$. Since $(\sE, \sF)$ is conservative, it coincides with its active reflected Dirichlet form; see \cite{CF}. Thus \eqref{e:4.6} holds. 
\end{proof}

\begin{example} \label{E:4.3} \rm (Symmetric mixed-stable-like process)  
  
Let $(\sX, d)$ be a locally compact separable metric space and $m$ be a Radon measure on $\sX$ with full support. 
Suppose that $(\sE, \sF)$ is a symmetric regular Dirichlet form on $L^2(\sX; m)$ of pure jump type with jump kernel
$J(dx,dy)= J(x, y) m(dx)m(dy)$; that is, 
$$
\sE (u, v) = \frac12 \int_{\sX\times \sX} (u(x)-u(y)) (v(x)-v(y)) J(x, y) m(dx) m(dy) 
\quad \hbox{for } u, v\in \sF.
$$
Note that $(J(x,y)m(dy),t)$ is a L\'evy system of the Hunt process $X$. 
Let $\phi:[0,\infty)\to [0,\infty)$ be a strictly increasing function with  $\phi(0)=0$ and $\phi(1)=1$
satisfying the condition  \eqref{eqnphi}. 
It is known that under {\rm VD} and {\rm RVD},  
\begin{equation*}\label{e:1.5} 
 {\bf HK}(\phi )\Longleftrightarrow {\bf J}_{\phi }+{\bf E}_{ \phi}
\Longleftrightarrow {\bf J}_{\phi }+{\bf (EP)}_{  \phi,\leq}.
\end{equation*}
 The first equivalence is established in  \cite[Theorem 1.13]{CKW} (though $\sX$ is assumed to be unbounded in \cite{CKW},   the results  there hold for bounded $\sX$ as well),  while the second equivalence follows from \cite[Lemma 2.7]{CKW} and Lemma \ref{lemma22}. Thus by Corollary \ref{C:1.7}, under $ {\bf HK}(\phi)$,  the uniform BHP holds with radius $R={\rm diam}(\sX)/(2\lambda_0)$ 
on any proper open set $D\subset \sX$,   where $\lambda_0\geq 2$ is the constant in RVD \eqref{e:rvd}.  

\medskip 

When  $ \beta <2$ in \eqref{eqnphi}, by \cite[Remark 1.14]{CKW}, we know that under VD and RVD,  
\[
{\bf HK}(\phi )\Longleftrightarrow {\bf J}_{\phi } .
\]
Hence by Corollary \ref{C:1.7}, under  VD, RVD and ${\bf J}_{\phi }$ with  $ \beta <2$,  the uniform BHP holds with radius $R={\rm diam}(\sX)/( 2 \lambda_0)$ on any proper open set $D\subset \sX$. 
\end{example}

\begin{example} \label{E:4.4} \rm
 The last paragraph of Example \ref{E:4.3} in particular implies the uniform BHP holds for the following three  families of discontinuous processes.
 
 \begin{enumerate} [(i)]
 \item (Stable-like non-local operator of divergence form on $\R^d$) 
 
 Suppose that $d\geq 1$,  $\alpha \in (0, 2)$ and $c(x, y)$ is a   symmetric function on $\R^d\times \R^d$  that is bounded between two positive constants.  Let 
 \begin{eqnarray*}
 \sF &=& \left\{ u\in L^2(\R^d; dx):  \int_{\R^d\times \R^d} \frac{ (u(x)-u(y)) ^2}{  |x-y|^{d+\alpha} }   dx dy <\infty \right\}, \\ 
 \sE (u, v) &=&  \frac12 \int_{\R^d \times \R^d } (u(x)-u(y)) (v(x)-v(y)) \frac{c(x, y)}{|x-y|^{d+\alpha}} dx dy  
\quad \hbox{for } u, v\in \sF.
   \end{eqnarray*}
   It is easy to see that  $(\sE, \sF)$  is a regular Dirichlet form on   $L^2(\R^d; dx)$.
   Its infinitesimal generator $\sL$ is a stable-like non-local operator of divergence form:
   $$
   \sL u (x)= \lim_{\eps \to 0}  \int_{\R^d  }  \1_{\{ |y-x|>\eps\}} (u(y)-u(x))  \frac{c(x, y)}{|x-y|^{d+\alpha}}   dy .  
 $$
 As a special case in \cite{CK}, we know  there is a Feller process $X$ on $\R^d$ associate with the regular Dirichlet form $(\sE, \sF)$ on
  $L^2(\R^d; dx)$ and ${\rm {\bf HK}}(\phi)$ holds for $X$ with $\phi (r)=r^\alpha$. Hence by Corollary \ref{C:1.7},
 the uniform BHP holds for $X$  with radius $R=\infty$ on any open subset of $\R^d$. 
 The Feller process $X$ is called a symmetric stable-like process on $\R^d$ in \cite{CK}. 
 
 \item (Isotropic unimodal L\'evy processes on $\R^d$)  
 
 Suppose that $X$ is a rotationally symmetric unimodal L\'evy process on $\R^d$ whose 
 L\'evy exponent $\xi \to \Psi (|\xi |)$ satisfies the following upper and lower scaling properties: there are constants
$0<\beta_1< \beta  <2 $ and $c>1$ so that
\begin{equation}\label{e:C1.2}
c^{-1} \left(\frac{R}{r} \right)^{\beta_1} \leq \frac{\Psi (R ) }{\Psi (r ) } \leq
c \left(\frac{R}{r} \right)^{ \beta }
\quad \hbox{for any } R\geq r>0.
\end{equation}
Condition \eqref{e:C1.2}
implies that the L\'evy process is  purely discontinuous, and by
\cite[Corollary 23]{BGR1},  its L\'evy intensity measure has density kernel 
$x \to j(|x|)$ which satisfies 
$$
 \frac{c^{-1} }{|x|^d \Phi (|x|)} \le j(|x|)   \le  \frac{c }{|x|^d \Phi (|x|)}  \quad \text{for all }x\neq0,
$$
where $\Phi  ( r) :=  \max_{|x| \le r} 1/\Psi (1/|x|) $.
It is easy to see that $\Phi (r)$ is comparable to $1/\Psi (1/r)$.
It follows from \eqref{e:C1.2}
that the
same two-sided estimates hold for $\Phi$ in place of $\Psi$. In other words, condition 
${\bf J}_{\phi }$ holds with  $ \beta <2$. Euclidean space $\R^d$ has infinite diameter, 
and when equipped with Lebesgue measure, is VD and RVD. So by Corollary \ref{C:1.7}, the uniform BHP  holds with radius $R=\infty$ 
on any proper open set $D\subset \R^d$.    This extends the BHP  result of \cite{KSV2}
 in two aspects: (a) from $R=1$ to $R=\infty$;  (b) The class of L\'evy processes here is more general
 without complete Bernstein function condition.  
 
 \medskip
 
 \item (Censored   stable-like processes on Ahlfors  $d$-regular subsets of   $\R^d$)
 
 Let $D\subset \R^d$ be an open Ahlfors  $d$-regular set, that is, there is a constant $c>0$ and $ r_D\in (0, \infty]$ so that   
 $m(D\cap B(x, r)) \geq c m(B(x, r))$ for every $x\in D$ and $r\in (0, r_D)$ where $m$ is the Lebesgue measure on $\R^d$. 
 Observe   that under this assumption, 
 $(D, m)$ equipped with the Euclidean metric has the VD and RVD properties for small $r>0$.  
 Examples of $d$-regular open sets include   Lipschitz open sets and some non-smooth domains such as snowflake domains. 
 Let $c(x, y)$ be a symmetric function defined on $D\times D$ that is bounded between two positive constants, and  $\alpha \in (0, 2)$ {  be} a constant. Define 
 \begin{eqnarray*}
 \sF &=& \left\{ u\in L^2(D; dx): \int_{D\times D} \frac{ (u(x)-u(y) )^2 }{ |x-y|^{d+\alpha} } dx dy<\infty \right\}, \\
 \sE(u, v) &=&\frac12  \int_{D\times D} c(x,y)\frac{(u(x)-u(y)) (v(x)-v(y))}{|x-y|^{d+\alpha} } dx dy
 \quad \hbox{for } u, v \in \sF.
 \end{eqnarray*}
 By \cite{CK}, $(\sE, \sF)$ is a regular Dirichlet form on $L^2(\overline D; dx)$ and its associated Hunt process $X$ has a jointly H\"older 
 continuous transition density function $p(t, x, y)$ on $(0, \infty)\times \overline D \times \overline D$ that has a two-sided estimates 
\[ 
p(t,x,y)\asymp t^{-d/\alpha}\wedge\frac{t}{|y-x|^{d+\alpha}}\quad\hbox{ for }t<r_D^\alpha,\,x,y\in \overline{D}. 
\]
 So $X$ can be refined to be a Feller process that can starting from every point in $\overline D$ and has strong Feller property.
 The infinitesimal generator of $X$ and $(\sE, \sF)$ on $L^2(\overline D; dx)$ is 
 $$
 \sL u(x)= \lim_{\eps \to 0} \int_D \1_{\{|y-x|>\eps\}} (u(y)-u(x)) \frac{c(x, y)}{|y-x|^{d+\alpha}} dy.
 $$

Clearly, {\bf (Jc)} and ${\bf (Jt)}_{\phi, \bar r}$  hold for $\bar{r}=r_D/\lambda_0$, where $\lambda_0\geq 2$ is the constant in RVD \eqref{e:rvd}, and ${\bf (EP)}_{\phi,r_0,\leq}$ holds for any $r_0>0$ by the heat kernel estimate and by a same proof as Proposition \ref{P:2.1}. By  Theorem \ref{T:1.3}, the uniform BHP for $X$ holds with radius $R=r_D/(2\lambda_0)$ on any (relative) open set of $\overline D$. In particular, the scale invariant BHP holds on the open set $D$ but it is possible that $\partial D$ can be invisible for $X$ for small $\alpha\in (0, 1)$, in which case the BHP on $D$ is just a consequence of the elliptic Harnack inequality on $\overline D$. We next discuss when $\partial D$ is visible for $X$. 
 
 Let $\sigma_{\partial D} :=\inf\{t>0: X_t \in \partial D\}$. 
 We say $\partial D$ is polar for $X$ if $\bP_x (\sigma_{\partial D}<\infty)=0$ for every $x\in \overline D$; otherwise, we say
 $\partial D$ is non-polar for $X$.   Denote by $X^0$  the  subprocess  $X$ killed upon leaving the open set $D$.
 When $c(x, y)$ is a constant, the processes $X$ and $X^0$ are  called, respectively, a  reflected $\alpha$-stable process on $\overline D$ and     a censored (or resurrected) $\alpha$-stable process in $D$.  It is shown in \cite[Theorem 2.5]{BBC} that  $\partial D$ is polar for $X$   if and only if $\partial D$ is polar with respect to the isotropic $\alpha$-stable process on $\R^d$,
 which is equivalent to $\partial D$ being of zero $W^{\alpha/2, 2}$-capacity. 
 Suppose now $\partial D$ is non-polar for $X$. 
 Then there are non-trivial  functions of $X$ that are regular harmonic in $B(\xi, r)\cap D$ and vanish on $B(\xi, r)\cap \partial D$
 for some $\xi \in \partial D$ and $r>0$. 
 As noted at the end of last paragraph, the scale-invariant BHP for $X$   holds  with radius $R= \diam (D)/(2\lambda_0)$ on the open set $D$.
 In particular, when $D$ is a  Lipschitz open set in $\R^d$, $\partial D$ is polar for $X$ if and only if $\alpha \in (0, 1]$
 (in this case $X^0=X$ is conservative), 
 and when $\alpha \in (1, 2)$, $X$ hits $\partial D$ in  finite time with positive probability and the scale invariant BHP for $X$ holds on the open set $D$
 with radius $R= \diam (D)/(2\lambda_0)$. This not only recovers but extends significantly  Theorem 1.2 of  \cite{BBC}, where the invariant BHP
 is established for censored $\alpha$-stable processes in bounded $C^{1,1}$-open sets for $\alpha \in (1, 2)$. 
   \end{enumerate} 
 \end{example}

 \begin{example}\label{E:4.5} \rm
 (Tempered  symmetric mixed stable-like processes in metric measure spaces)

Let $(\sX, d)$ be a locally compact separable metric space and $m$ be a Radon measure on $(\sX, d)$ with full support. 
Suppose that there is an increasing function $V(r)$ so that   there exist constants $c_2>c_1>1$
\begin{equation} \label{e:4.9a}
c_1V(r)\leq V(2r)\leq c_2 V(r)  \quad \hbox{and} \quad 
m(B(x, r)) \asymp V(r)  \quad \hbox{  for all  } x \in \sX  \hbox{ and }  r>0. 
\end{equation}
Suppose that $m(\sX)=\infty$ and $(\sE, \sF)$ is a symmetric regular Dirichlet form on $L^2(\sX; m)$ of pure jump type with jump kernel
$J(dx,dy)= J(x, y) m(dx)m(dy)$; that is, 
$$
\sE (u, v) = \frac12 \int_{\sX\times \sX} (u(x)-u(y)) (v(x)-v(y)) J(x, y) m(dx) m(dy) 
\quad \hbox{for } u, v\in \sF.
$$
Note that $(J(x,y)m(dy),t)$ is a L\'evy system of the Hunt process $X$. 
Suppose that  $c(x, y)$ on $\sX\times \sX$  is a measurable symmetric function  that is bounded between two positive constants,
$\phi$ is an increasing function on $[0, \infty)$  satisfying condition  \eqref{eqnphi}  with $0<\beta_1\leq  \beta <2$, and $\lambda  >0$ and 
$\beta \in (0, 1]$ are constants so that 
\begin{equation}
J(x, y) = \frac{c(x, y)} {V(d(x, y)) \phi (d(x, y)) \exp ( \lambda d(x, y)^\beta )} \quad \hbox{for } x\neq y \in \sX.
\end{equation} 
It is easy to check (see Remark \ref{R:1.8}(i)) that $J$ satisfies conditions    ${\bf (Jc)}_r$  and ${\rm \bf (Jt)}_{\phi, \bar r}$ for any $r, \bar r\in (0, \infty)$. On the other hand, by \cite[Theorem 1.2]{CK2} and \cite[Remark 1.3]{CKK},  the regular Dirichlet form 
$(\sE, \sF)$ on $L^2(\sX; m)$ has a conservative Feller process $X$ associated with it, which has 
 a jointly continuous heat kernel $p(t, x, y)$ and  there are positive constants $ c_3\geq c_4>0$ and $C\geq 1$
so that 
\begin{eqnarray}\label{e:4.10}
&& \  C^{-1}\left( \frac{1}{V( \phi^{-1}(t))}   \wedge \frac{t}{V(d(x, y)) \phi(d(x, y))\exp(  c_3  d(x, y)^\beta)}  \right)  \nonumber \\
&\leq &  p(t, x, y) \,  \leq \, C \left( \frac{1}{V( \phi^{-1}(t))}   \wedge \frac{t}{V(d(x, y)) \phi(d(x, y)) \exp ( c_4  d(x, y)^\beta )} \right) 
\end{eqnarray} 
for every $t\in (0, 1]$ and $x, y\in \sX$. 
  Thus $ {\bf (EP)}_{\phi,\leq, 1/2}$ holds by the same argument as that for  \cite[Lemma 2.7]{CKW},
and so  ${\bf (EP)}_{\phi,r_0,\leq}$ holds for every $r_0>0$   by Proposition \ref{P:2.1}(a). 
We conclude from   Theorem \ref{T:1.3}  that  for any $R>0$, the uniform BHP for $X$ holds with radius $R $ on any (relative) open set of $\sX$.  

Using the result from  \cite{CKW}, we can in fact drop the global comparability assumption  \eqref{e:4.9a},  replace it by
the   VD and RVD assumptions  on $m$, and by a similar argument as that in \cite{CK2} to 
 get the two-sided heat kernel estimates \eqref{e:4.10} but with $V(x, d(x, y))$ in place of 
$V(d(x, y))$; see  \cite{CKKW} for a similar result   for  symmetric diffusions with tempered  jumps.  
 With almost the same proof of Lemma \ref{lemma24}, we can also check that $J$ satisfies conditions ${\bf (Jc)}_r$  and ${\rm \bf (Jt)}_{\phi, \bar r}$ for any $r, \bar r\in (0, \infty)$.
 Consequently,    by Theorem \ref{T:1.3},  the uniform BHP  for $X$ holds with radius $R $ on any (relative) open set of $\sX$
for any $R>0$. 
\end{example}

\begin{example}\label{E:4.6} \rm
 (Stable-like subordinate diffusions in metric measure spaces).
 
  Suppose that $(\sX,d,m)$ is an Ahlfors regular $d_f$-space for some $d_f>0$. Suppose that there is an $m$-symmetric diffusion process $Z_t$ on $\sX$ with 
a continuous transition density $p^Z(t,x,y)$ satisfying the sub-Gaussian bounds
\begin{equation}\label{e:4.8}
\frac{c_1}{t^{d_f /d_w}}\exp \left(-c_2\big(\frac{d(x,y)^{d_w}}t\big)^{1/(d_w-1)} \right)
\leq p^Z(t,x,y)\leq \frac{c_3}{t^{d_f/d_w}}\exp \left(-c_4\big(\frac{d(x,y)^{d_w}}t\big)^{1/(d_w-1)}\right), 
\end{equation}
where $d_w\geq 2$ is called the walk dimension. When $d_w=2$, the above is the classical Aronson's Gaussian estimate. 
It is known that the heat kernel estimates \eqref{e:4.8} hold for  $d_w>2$ on some fractal type spaces, such as
 the Sierpi\'nski gasket \cite{BaBa}, nested fractals \cite{FK,Ku}, or the generalized Sierpinski carpets \cite{BaBa2}. For $0<s<1$, let $\{S_t\}_{t>0}$ be an $s$-stable subordinator that is independent of $Z$.
The subordinate diffusion  $Z_{S_t}$ is a symmetric Feller  process of pure jump type that satisfying {\bf HK}$(r^{sd_w})$.  Hence, the uniform BHP holds by Corollary \ref{C:1.7}. 
Moreover, by the stability theorem of \cite{CKW}, {\bf HK}$(r^{sd_w})$ holds 
for any regular Dirichlet form $(\sE, \sF)$ on $ L^2(\sX;m)$ of pure jump type whose jump  kernel is comparable to that of the subordinate diffusion $\{X_t:=Z_{S_t}; t\geq 0\}$, that is, satisfying the condition ${\bf J}_\phi$ with $\phi (r)=r^{sd_w}$, then ${\bf HK} (r^{sd_w})$ holds for $(\sE, \sF)$ and so does the uniform BHP by Corollary \ref{C:1.7}. 
\end{example}

\medskip

\begin{example}\label{E:4.7} \rm
 (Stable-like non-local operators of non-divergence form on $\R^d$).

Let $d\geq 1$ and $0<\alpha<2$. 
Let $\kappa_\alpha (x, z)$ be a measurable function on $\R^d\times \R^d$ that is bounded between two positive constants.
When $\alpha=1$, we assume that for every $x\in \R^d$ and every $0<r<R<\infty$, 
\begin{equation}\label{e:4.9}
\int_{r<|z|<R} \frac{z \kappa_1 (x, z)}{|z|^{d+1}} dz =0
\end{equation} 
The above  condition is clearly satisfied if $\kappa_1 (x, z)$ is symmetric in $z$, that is, if 
$\kappa _1(x, z)=\kappa_1 (x, -z)$ for all $x, z\in \R^d$. For $0<\alpha<2$ and a bounded $\R^d$-valued function $b$ on $\R^d$, 
consider the following generator acting on $C^2_c(\R^d)$: 
\begin{equation}\label{e:sL}
 \sL f(x):=
   \int_{\R^d} \left(f(x+z)-f(x)-z^{(\alpha)}\cdot\nabla f(x)\right)\frac{\kappa_\alpha (x,z)}{|z|^{d+\alpha}   }
	d z
	+ \1_{\{1<\alpha<2\}} b(x) \cdot \nabla f (x)
\end{equation} 
where 
$$
z^{(\alpha)}= 
\begin{cases} 
0 \quad & \hbox{if } 0<\alpha <1, \\
z \1_{\{ |z|\leq 1\}} \quad & \hbox{if } \alpha =1, \\
z \quad & \hbox{if } \alpha >1.
\end{cases}
$$

Suppose that $X=(X_t, t\geq 0; \bP_x, x\in \R^d)$ is a conservative Hunt process on $\R^d$
having generator $\sL$ in the sense that for every $x\in \R^d$ and $f\in C_c^2(\R^d)$,
\begin{equation} \label{e:4.11a}
M^f_t:= f(X_t) -f(X_0)-\int_0^t \sL f(X_s) ds
\end{equation}
is a martingale under $\bP_x$.
As a special case of \cite[Theorems 1.1 and 1.5]{CZ1}, we know  that when $\kappa(x, z)$ is H\"older-continuous in $x$,
 there is a Feller process $X$ having $\sL$ as its generator having infinite lifetime.

It can be shown that  $X$ has L\'evy system $(\frac{\kappa_\alpha (x,  y-x)}{|y-x|^{d+\alpha}} dy, t)$;
see, e.g.,   \cite[Lemma 4.9]{CZ2}. Thus condition ${\bf J}_\phi$ holds with $\phi (r)=r^\alpha$ 
for $X$ with $m$ being the Lebesgue measure on $\R^d$. 
We now check that condition ${\bf (EP)}_{\phi, \leq }$ holds for $X$ as well when the drift $b=0$. 
Let $\psi \in  C^2_c(\R^d)$ so that $0\leq \psi \leq 1$ on $\R^d$, $\psi (0)=1$ and $\psi (x) =0$ for $|x|\geq 1$. 
For each $x_0\in \R^d$ and $r>0$, let $f(x)=\psi  ((x-x_0)/r)$. Then
\begin{eqnarray}\label{e:4.11}
\bP_{x_0} (\tau_{B(x_0, r)} \leq t) 
&\leq & \bE_{x_0} \left[ (1-f)(X_{\tau_{B(x_0, r)} \wedge t}) - (1-f)(X_0) \right] \nonumber  \\
&= &  - \bE_{x_0} \int_0^{\tau_{B(x_0, r)} \wedge t} \sL f(X_s) ds.
\end{eqnarray}
In view of the assumption \eqref{e:4.9} when $\alpha=1$, we have for any $0<\alpha<2$, 
\begin{eqnarray}
&& |\sL f(x)| \nonumber \\
&=& \Big| \int_{\R^d} \left(\psi (\frac{x+z-x_0}r)-\psi(\frac{x-x_0}r)-r^{-1}z^{(\alpha)}\cdot (\nabla \psi) (\frac{x-x_0}r)\right)\frac{\kappa(x,z)}{|z|^{d+\alpha}  }d  z  \nonumber \\
&& + \1_{1<\alpha<2} \, r^{-1} b(x) (\nabla \psi) (\frac{x-x_0}{r}) \Big|\nonumber \\
&\leq & c_\psi \int_{\R^d}  \left( \1_{0<\alpha<1} 
(\frac{|z|}{r} \wedge 1)  + \1_{\alpha=1}  (\frac{|z|^2}{r^2}\wedge 1)   +
\1_{1<\alpha<2}  \left( \frac{|z|^2}{r^2} \wedge \frac{|z|}{r} \right)\right) \frac{\kappa_\alpha (x, z)}{|z|^{d+\alpha} }d  z 
\nonumber \\
&&  + \1_{1<\alpha<2} \| b\|_\infty \|\nabla \psi \|_\infty /r  \nonumber \\
&=& c_\psi \int_{\R^d} \left(  \1_{0<\alpha<1} (|w|\wedge 1) + \1_{\alpha=1}  (|w|^2 \wedge 1)+
\1_{1<\alpha<2}  (|w|^2\wedge |w| ) \right)
\frac{\kappa_\alpha (x, rw)}{r^\alpha |w|^{d+\alpha} }d  w \nonumber \\
&& + \1_{1<\alpha<2} \| b\|_\infty \|\nabla \psi \|_\infty /r   \nonumber \\
&\leq & c_\psi c_\alpha /r^\alpha  + \1_{1<\alpha<2} \| b\|_\infty \|\nabla \psi \|_\infty /r , \label{e:4.12}
\end{eqnarray}
where in the second equality, we used a change of variable $z=rw$. 
Hence we conclude from \eqref{e:4.11}-\eqref{e:4.12} that when $b=0$, ${\bf (EP)}_{\phi, \leq} $ holds for $X$
and so the BHP holds for $X$ on any non-trivial open subset of $\R^d$.
Note that when $1<\alpha<2$ , for each $r_0>0$,  $1/r \leq r_0^{\alpha-1}/r^\alpha$
for every $r\in (0, r_0]$. Thus when there is a non-trivial drift for $\alpha \in (1, 2)$, we deduce from \eqref{e:4.11}-\eqref{e:4.12}  that for every $r_0>0$,
${\bf (EP)}_{\phi,r_0,\leq}$ holds with some constant $C$ depending on $r_0$. So by Corollary \ref{C:1.5}, the uniform BHP holds for $X$ 
 with radius $R=r_0$ on any non-trivial open subset of $\R^d$.  We point out that in this example, properties ${\bf J}_{\phi }$, ${\bf (EP)}_{\phi, \leq} $,
${\bf (EP)}_{r_0,\phi,\leq}$ and hence the BHP  depend only 
on the lower and upper bounds of the function $\kappa (x, z)$ and the upper bound of the drift function $b(x)$.
This example   significantly extends the BHP result in \cite{RS}  not only  from a non scale invariant BHP to the uniform BHP but also 
without assuming the function $\kappa_\alpha (x, z) $ being symmetric in $z$. 
\end{example}

\begin{example} \label{E:4.8} \rm
 (Mixed stable-like non-local operators of non-divergence form on $\R^d$).
 
 The above example can be extended to the mixed stable-like non-local operators of non-divergence form  
  considered in \cite{CZ2}. Suppose that $\phi$ is a strictly increasing function on $[0, \infty)$ that satisfies \eqref{eqnphi}. Consider the following generator  
\begin{equation}\label{e:sL2}
 \sL f(x):=
   \int_{\R^d} \left(f(x+z)-f(x)-z^{(\phi)}\cdot\nabla f(x)\right)\frac{\kappa  (x,z)}{|z|^d \phi (|z|) }d  z,
\end{equation} 
where,  as specified by  \cite[(1.4)]{CZ2},
\[
z^{(\phi)}  :=\begin{cases} 
0 \quad & \hbox{ for Case }1:\,\int_0^1\frac{dr}{\phi(r)}<\infty,  \\
z \1_{\{ |z|\leq 1\}} \quad & \hbox{ for Case }2:\,\int_0^1\frac{dr}{\phi(r)}=\infty\hbox{ and }\int_1^\infty\frac{dr}{\phi(r)}=\infty,\\
z \quad & \hbox{ for Case } 3:\,\int_0^1\frac{dr}{\phi(r)}=\infty\hbox{ and }\int_1^\infty\frac{dr}{\phi(r)}<\infty,
\end{cases}
\]
and $\kappa  (x,z)$ is a measurable function on $\R^d\times\R^d$ that is bounded between two positive functions, and for Case 2, we 
assume in addition that $\int_{r<|z|<R}z\kappa(x,z)dz=0$ for every $0<r<R<\infty$. We further assume the condition ($\bf A_\phi^{(1)}$) of \cite{CZ2},
\begin{equation}\label{e:4.14}
  c_\phi: =\sup_{r\in (0,1]}\int_0^\infty \left(\1_{\hbox{Case 1}}(s\wedge 1)+\1_{\hbox{Case 2}} (s^2 \wedge 1)+\1_{\hbox{Case 3}}(s^2\wedge s ) \right)\frac{\phi(r)ds}{s\phi(rs)}<\infty. 
\end{equation}
Suppose that $X$ is a conservative Hunt process on $\R^d$
having $\sL$ as its generator in the sense of \eqref{e:4.11a}. 
By the proof of  \cite[Lemma 4.9]{CZ2}, $X$ has L\'evy system $ \big( \frac{\kappa  (x, y-x)}{|y-x|^d \phi (|y-x|)}dy, t  \big)$.
Thus condition ${\bf J}_\phi$ holds. Moreover,  ${\bf (EP)}_{1, \phi,\leq}$ holds by \eqref{e:4.11}, \eqref{e:4.14} and the following estimate
for $f(x):=\psi  ((x-x_0)/r)$   with $r\leq 1$: 
\begin{eqnarray}
	&& |\sL   f  (x)| \nonumber \\
	&=& \Big| \int_{\R^d} \left(\psi \left(\frac{x+z-x_0}r \right)-\psi \left(\frac{x-x_0}r \right) -r^{-1}z^{ (\phi)}\cdot 
	\nabla \psi \left(\frac{x-x_0}r\right) \right)\frac{\kappa(x,z)}{|z|^d\phi(z)}dz  \nonumber \Big|\\
	&\leq & c_\psi \int_{\R^d}  \left( \1_{\hbox{Case 1}} 
	\left( \frac{|z|}{r} \wedge 1 \right)  + \1_{\hbox{Case 2}}  \left(\frac{|z|^2}{r^2}\wedge 1 \right)   +
	\1_{\hbox{Case 3}}\left(\frac{|z|^2}{r^2} \wedge \frac{|z|}{r} \right)\right) \frac{\kappa(x,z)}{|z|^d\phi(|z|)}dz 
	\nonumber \\
	&\leq & c_\psi \int_{\R^d}\left(\1_{\hbox{Case 1}}(|w|\wedge 1)+\1_{\hbox{Case 2}} (|w|^2 \wedge 1)+\1_{\hbox{Case 3}}(|w|^2\wedge |w| ) \right)
	\frac{\|\kappa\|_\infty}{|w|^{d}\phi(r|w|)}dw \nonumber\\
	&=&c_\psi c_d\|\kappa\|_\infty\int_0^\infty \left(\1_{\hbox{Case 1}}(s\wedge 1)+\1_{\hbox{Case 2}} (s^2 \wedge 1)+\1_{\hbox{Case 3}}(s^2\wedge s ) \right)\frac{ds}{s\phi(rs)} \nonumber\\
	&\leq&\frac{ c_\psi c_d\|\kappa\|_\infty}  {\phi(r)} 
	 \sup_{r\in (0,  1]}\int_0^\infty \left(\1_{\hbox{Case 1}}(s\wedge 1)+\1_{\hbox{Case 2}} (s^2 \wedge 1)+\1_{\hbox{Case 3}}(s^2\wedge s ) \right)\frac{\phi(r)ds}{s\phi(rs)}  \nonumber\\ 
	 & =&\frac{  c_\phi  c_\psi c_d\|\kappa\|_\infty}  {\phi(r)} .   \label{e:4.20} 
\end{eqnarray}
Consequently, by Corollary \ref{C:1.5}, the uniform BHP holds for $X$   with radius $R=1$ on any non-trivial open subset of $\R^d$. Moreover, if the estimate of \eqref{e:4.14} holds with $r\in (0,1]$ replaced with $r\in (0,\infty)$, ${\bf (EP)}_{\phi,\leq}$ holds.  In this case, by Corollary \ref{C:1.5},  the uniform  BHP holds for $X$  with radius $R=\infty$ 
on any non-trivial open subset of $\R^d$. We emphasize that the comparison constant in these invariant BHP depends on the coefficient $\kappa (x, z)$ only through its upper bound not on its modulo of continuity.\medskip

Now we give a sufficient condition on the non-local operator $\sL$ of \eqref{e:sL2} so that there is a   conservative Hunt process $X$ on $\R^d$
associated with it in the sense of \eqref{e:4.11a}.  
 Following \cite[Definition 1.2]{CZ2}, we call a strictly increasing continuous function
 $\ell: (0,1]\to(0,\infty)$ a Dini function if 
$$
\lim_{t\to 0}\ell(\lambda t)/\ell(t)=1 \quad \hbox{for every } \lambda>0
$$
and
$\int^1_0\frac{\ell(t)}{t} d t<\infty. $
Suppose that the function 
 $\kappa (x, z)$ in \eqref{e:sL2} satisfies that for some $\kappa_0\geq 1$ and Dini function $\ell$
so that for all $x, y, z\in \R^d$ with $|x-y|\leq 1$, 
\begin{align}\label{Con1}
\kappa^{-1}_0\leq \kappa(x,z)\leq \kappa_0 \quad \hbox{and} \quad 
 |\kappa(x,z)-\kappa(y,z)|\leq \ell^2(|x-y|).
\end{align}
As a special case of \cite[Theorem 1.4]{CZ2}, 
under condition \eqref{e:4.14} and 
\begin{align}\label{Con3}
 \int^t_0\frac{1}{r}\left(\frac{\ell(r)}{\ell(t)}+\frac{\phi(r)}{\phi(t)}\right) d \phi(r)<\infty \quad \hbox{for } t\in(0,1], 
\end{align} 
there is a conservative
Feller process having generator $\sL$. Moreover, ${\bf HK}(\phi, 1)$ holds. Consequently,  the uniform  BHP holds for such $\sL$ with radius $R=1$ on any non-trivial open subset of $\R^d$.
\end{example}

\begin{example} \label{E:4.9} \rm
 (Tempered mixed stable-like non-symmetric non-local operators  on $\R^d$).
 
 The last example can be extended to allow the jump kernel having  exponential decay,
 that is, to non-local operator
 \begin{equation}\label{e:4.22}
 \sL f(x):=
   \int_{\R^d} \left(f(x+z)-f(x)-z^{(\phi)}\cdot\nabla f(x)\right)\frac{\kappa(x,z)}{|z|^d \phi (|z|) \exp (\lambda |z|^\beta )}d  z,
\end{equation} 
where $\lambda >0$ and $\beta \in (0, 1]$.  Here $\phi$ and $\kappa (x, z)$ satisfy the same condition as in Example \ref{E:4.8}. 
 Suppose that $X$ is a conservative Hunt process on $\R^d$
having $\sL$ as its generator in the sense of \eqref{e:4.11a}. 
By the proof of  \cite[Lemma 4.9]{CZ2}, $X$ has L\'evy system 
$\big(\frac{\kappa  (x, y-x)}{|y-x|^d \phi (|y-x|) \exp (\lambda |y-z|^\beta )}  dy, t \big) $.
 It is easy to check (see Remark \ref{R:1.8}(i)) that $J$ satisfies conditions    ${\bf (Jc)}_r$  and ${\rm \bf (Jt)}_{\phi, \bar r}$
for any $r, \bar r\in (0, \infty)$.   
Moreover,  the estimate $\sL f(x)\leq \frac{c_\phi c_dc_\psi  \|\kappa\|_\infty}  {\phi(r)}$ holds by the same calculation as in \eqref{e:4.20}
for  $f(x):=\psi  ((x-x_0)/r)$   with $r\leq 1$, and so   ${\bf (EP)}_{\phi,1,\leq}$ holds by \eqref{e:4.11} and \eqref{e:4.14}. 
We conclude from Theorem \ref{T:1.3} that   the uniform BHP for $\sL$ holds with radius $R=1$ on any   open set of $\sX$.  

Under the same Dini condition \eqref{Con1}-\eqref{Con3} on $\kappa (x, z)$, there is a conservative
Feller process $X$ having  $\sL$ as its generator.  Indeed, as we see from Example \ref{E:4.8}, there is a 
conservative Feller process $\wt X$ having  generator 
 $$
 \wt \sL f(x):=
   \int_{\R^d} \left(f(x+z)-f(x)-z^{(\phi)}\cdot\nabla f(x)\right)\frac{\kappa(x,z)}{|z|^d \phi (|z|) \exp (\lambda (|z|\wedge 1)^\beta )}d  z,
$$
One can then construct a  conservative
Feller process $X$ having  $\sL$ as its generator from $\wt X$ through removing and adding jumps.  See
\cite{Ba, INW} as well as  \cite[\S 5.2]{CZZ} for details. 
\end{example}

\begin{example}\label{E:geostable}\rm  (Geometric  stable-like processes  on $\R^d$). 

In this example, we give a class of non-symmetric non-local operators that 
satisfy the conditions  $({\bf Jc})_{\bar{r}}$,  ${\rm \bf (Jt)}_{\phi,\bar r}$ and ${\bf(EP)}_{\phi,r_0,\leq}$  for some $\bar{r},r_0\in(0,\infty]$
and so the uniform BHP holds on any open sets but $\phi$ is not reverse doubling. 
This class of non-local operators includes the generators of geometric stable processes on $\R^d$. 
Let $d\geq 1$ and $\alpha \in (0, 2]$. 
A geometric $\alpha$-stable process is an isotropic  L\'evy process $X_t$ on $\R^d$ with  L\'evy exponent  $\Psi(\xi)=\log(1+|\xi|^{\alpha/2})$; that is,
$$
\bE_x \left[ e^{ i \xi \cdot ( X_t-X_0) } \right] = e^{-t \Psi  (\xi) } \quad \hbox{for every } \xi \in \R^d.
$$
 It is known the L\'evy measure of $X$ has a density $j_0( |z|)$ with respect to the Lebesgue measure on $\R^d$,
 and the infinitesimal generator $\sL_0$ of $X$ is given by
 $$
  	\mathcal{L}_0 f(x):=    \int_{\R^d} (f(x+z) - f(x))j_0 (|z|)dz\quad\hbox{ for }x\in\R^d\hbox{ and }f\in C_c^2(\R^d),
$$
 By \cite[Theorems 3.4 and 3.6]{SSV}, $j_0(r)$ is comparable to $j(r)$ on $(0, \infty)$, where 
\begin{equation}
	j(r) :=\frac{1}{r^d}\wedge \frac{1}{ r^{d+\alpha}} . 
\end{equation}

	We now consider a strong Markov process $X$ with its generator $\mathcal{L}$ given by 
	\[
	\mathcal{L}f(x):= \int_{\R^d} (f(x+z) - f(x))\kappa(x,z)j(|z|)dz\quad\hbox{ for }x\in\R^d\hbox{ and }f\in C_c^2(\R^d),
	\] 
	where $\kappa (x, z)$ is bounded between two positive constants and symmetric in $z$.  
	By a proof similar to that  of  \cite[Lemma 4.9]{CZ2}, one can show that 
	  $(\kappa(x,y-x)j(|y-z|)dy,t)$ is a L\'evy system of $X$. Note that this class of processes include the geometric stable processes.

	\smallskip 
	
	We next show that the uniform BHP holds on any open domains with Radius $R=\infty$ when $\alpha\in(0, 2)$, and holds on any open domains with Radius $R$ for any $R\geq 1$ when  $\alpha=2$. 
	
    First, ${\bf (Jc)}$ holds by Remark \ref{R:1.8}(ii). One  can easily check that ${\bf (Jt)}_{\phi}$ holds with  
	\[
	\phi(r) :=\frac1{1-\log(r\wedge1)}\wedge r^\alpha, \quad r>0.
	\]	
	Observe that the increasing function $\phi$ has doubling property  but  is not reverse doubling. 
	Finally, we  show ${\bf EP}_{\phi,\leq}$ with \eqref{e:4.11a} for $\alpha<2$. Let $\psi \in  C^2_c(\R^d)$ so that $0\leq \psi \leq 1$ on $\R^d$, $\psi (0)=1$ and $\psi (x) =0$ for $|x|\geq 1$. 
	For each $x_0\in \R^d$ and $r>0$, let $f(x)=\psi  ((x-x_0)/r)$.
	Note that for $x\in\R^d$,  
 	\begin{align}
		  |\mathcal{L}f(x)| 
		&\leq  \frac{\|\kappa\|_\infty}{2} \int_{\R^d}  |f(x+z)+f(x-z)-2f(x)|j(|z|)dz  \nonumber \\
		&\leq c_\psi\|\kappa\|_\infty \int_{\R^d} \left(  (|z|^2/r^2) \wedge 1\right) j(|z|) dz . \label{e:4.25a}
		\end{align}
	So when $r\in (0, 1]$, 
\begin{align*}
  |\mathcal{L}f(x)| 
   &\leq   c_\psi\|\kappa\|_\infty \left(  \int_{0<|z|<r}  \frac{|z|^2}{r^2|z|^d} dz + \int_{r\leq |z|<1} \frac{1}{|z|^d}  dz  + \int_{|z|>1} \frac{1}{|z|^{d+\alpha}} dz\right)\\
   &\leq c_dc_\psi\|\kappa\|_\infty\left(\int_0^r\frac{2 s}{r^2}ds + \int_r^1\frac1sds + \int_1^\infty \frac{1}{s^{1+\alpha}} ds\right)\\
   &=c_dc_\psi\|\kappa\|_\infty(1-\log r+\alpha^{-1})\\
   &\leq c_2 /\phi(r); 
\end{align*}
while for $r\in (1, \infty)$, 	
\begin{align*}
  |\mathcal{L}f(x)| 
   &\leq   c_\psi\|\kappa\|_\infty \left(   \int_{0<|z|\leq 1}  \frac{|z|^2}{r^2 |z|^d} dz   + \int_{1<|z|\leq r } \frac{|z|^2}{r^2|z|^{d+\alpha}} dz
   + \int_{ |z|>  r } \frac{1}{ |z|^{d+\alpha}} dz 
   \right)\\
&\leq \begin{cases} 
c_{d, \alpha} c_\psi / r^\alpha = c_3/\phi (r)  &\hbox{for } \alpha \in (0, 2), \\
c_{d, 2} c_\psi    (1+ \ln r ) / r^2 &\hbox{for } \alpha =2.  \\
\end{cases}
\end{align*}
In summary,  we have    
$$
|\mathcal{L}f (x) |  \leq 
 \begin{cases} 
 c_4/\phi (r)  &\hbox{for  all } r>0 \hbox{ when } \alpha \in (0, 2), \\
c_4/\phi (r)  &  \hbox{for  all } r\in (0, r_0] \hbox{ when } \alpha=2.    \\
\end{cases} 
$$
 Thus we have by \eqref{e:4.11} that ${\bf EP}_{\phi,\leq}$ holds when $\alpha \in (0, 2)$
 and ${\bf EP}_{\phi, r_0, \leq}$ holds for every $r_0\geq 1$.  
Thus by Theorem \ref{T:1.3}, the uniform BHP holds with radius $R=\infty$ for $X$ when $\alpha \in (0, 2)$,
and  uniform BHP holds with radius $R$ for $X$ for any $R\geq 1$ when $\alpha=2$.

This substantially extends the scale invariant BHP for geometric stable processes on $\R^d$ obtained in \cite{GR} for intervals on $\R$ when $d=1$ and in \cite{KSV4} for unbounded domain in $\R^d$ at infinity when $d\geq 2$. 
\end{example}

 \begin{example} \label{E:4.10} \rm
(SDEs driven by L\'evy processes  on $\R^d$). 

Let  $Z$ be a pure jump isotropic L\'evy process on $\R^d$ with L\'evy measure $j(|z|) dz$  
with 
\begin{equation} \label{e:4.24} 
 \frac{c_1^{-1} }{|x|^d \Phi (|x|)} \le j(|x|)   \le  \frac{c_1 }{|x|^d \Phi (|x|)}  \quad \text{for all }x\neq0.
\end{equation} 
for some constant $c_1>1$.
Here $\Phi$  is an increasing function   on $[0, \infty)$ satisfying 
\begin{equation}\label{e:4.25}
c_2^{-1} \left(\frac{R}{r} \right)^{\beta_1} \leq \frac{\Phi (R ) }{\Phi (r ) } \leq
c_2\left(\frac{R}{r} \right)^{ \beta }
\quad \hbox{for any } R\geq r>0.
\end{equation}
for some constants  $c_2>0$ and  $0<\beta_1\leq  \beta <2$.

Suppose that $X$ is  a strong Markov process on $\R^d$ that satisfies 
$$
dX_t = \sigma (X_{t-}) dZ_t,
$$
 where $\sigma(x)$ is a $d\times d$ matrix-valued function on $\R^d$ that is uniformly elliptic and bounded. 
For $t>0$, denote    $\Delta Z_t:=Z_t-Z_{t-}$ and  $\Delta X_t:=X_t-X_{t-}$. 
Define 
$$
\bar Z_t = Z_t -\sum_{0<s\leq t}  \Delta Z_s  {\mathbbm 1}_{\{|\Delta_s|>1\}},
$$
 which is an isotropic L\'evy process on $\R^d$ with
L\'evy measure ${\mathbbm 1}_{\{ |z|\leq 1\}}  j(|z|) dz$. In particular, $\bar Z$ is an
$\R^d$-valued  square-integrable martingale.
 By Ito's formula, for $f\in C^2_b(\R^d)$,
 \begin{eqnarray*}
 f(X_t)- f(X_0)  &=&   \int_0^t \nabla  f(X_{s-}) \cdot dX_s
+\sum_{0<s\leq t} \Big( f(X_s)-f(X_{s-})-\nabla  f(X_{s-}) 
  \cdot \Delta X_s \Big)\\
&=& \int_0^t \nabla  f(X_{s-}) \sigma (X_{s-}) dZ_s
+\sum_{0< s\leq t} \Big( f(X_{s-}+\sigma (X_{s-}) \Delta Z_s) \\
&& \quad   -f(X_{s-})-\nabla  f(X_{s-})\cdot  \sigma (X_{s-})\Delta Z_s \Big)\\
&=& \int_0^t \nabla  f(X_{s-}) \sigma (X_{s-}) d\bar  Z_s
+\sum_{0< s \leq t} \Big( f(X_{s-}+\sigma (X_{s-})\Delta Z_s) \\
&& \quad  -f(X_{s-})-\nabla  f(X_{s-})\cdot
\sigma (X_{s-})\Delta Z_s \1_{\{|\Delta Z_s|\leq 1\}} \Big) .\\
&=&  {\rm martingale} + \int_0^t \sL f (X_s) ds,
  \end{eqnarray*}
 where we used the L\'evy system of $Z$ in the last equality  with  
 $$
 \sL f(x):= \int_{\R^d} \left( f(x+\sigma (x)z) -f(x) -\nabla f(x) \cdot \sigma (x) z {\mathbbm 1}_{\{|z|\leq 1\}}\right) j(|z|) dz.
$$
This shows that the strong Markov process has generator $\sL$. By a change of variable $w=\sigma (x) z$,
\begin{eqnarray*}
\sL f(x)&:=&  \int_{\R^d} \left( f(x+w) -f(x) -\nabla f(x) \cdot w {\mathbbm 1}_{\{| \sigma (x)^{-1} w|\leq 1\}}\right) j(| \sigma(x)^{-1}w |)
{\rm det}(\sigma(x)^{-1} ) dw \\
&=& \int_{\R^d} \left( f(x+w) -f(x) -\nabla f(x) \cdot w {\mathbbm 1}_{\{|   w|\leq 1\}}\right) j(| \sigma(x)^{-1}w |)
{\rm det}(\sigma(x)^{-1} ) dw .
\end{eqnarray*}
Here $\sigma (x)^{-1}$ is the inverse matrix of $\sigma (x)$ and ${\rm det}(\sigma(x)^{-1} )$ is the determinant of $\sigma (x)^{-1}$. 
Define
$$
 \kappa (x, w):= \frac{ j(| \sigma(x)^{-1}w|) {\rm det}(\sigma(x)^{-1} )}{|w|^d \Phi (|w|)}  .
$$
Since $\sigma (x)$ is uniformly elliptic and bounded, it follows from \eqref{e:4.24} and \eqref{e:4.25}
that $\kappa (x, w)$ is bounded between two positive
constants and $\kappa (x, w)= \kappa (x, -w)$.  By the proof of  \cite[Lemma 4.9]{CZ2}, the strong Markov process
$X$ has L\'evy system $ \big( \frac{\kappa  (x, y-x)}{|y-x|^d \Phi (|y-x|)}dy, t  \big)$ so condition ${\bf J}_\Phi$ holds.
Moreover, by a similar argument as that in Example \ref{E:4.8} via \eqref{e:4.20}, 
one can show that ${\bf (EP)}_{\phi,1,\leq}$ holds. Hence we have by Corollary \ref{C:1.5} that 
the uniform BHP holds for $X$   with radius $R=1$
on any non-trivial open subset of $\R^d$. 

When $\Phi (r)=r^\alpha$ for some $0<\alpha<2$, condition ${\bf (EP)}_{ \phi,\leq}$ holds as 
  the estimate of \eqref{e:4.14} holds with $r\in (0,1]$ replaced with $r\in (0,\infty)$.
Consequently  we have by Corollary \ref{C:1.5} that 
the uniform BHP holds for $X$   with radius $R=\infty$
on any non-trivial open subset of $\R^d$. 
In particular, this is the case when  $Z$ is an isotropic $\alpha$-stable process on $\R^d$. 
  \end{example}

\vskip 0.2truein 

\end{document}